\documentclass[11pt,reqno,a4paper]{amsart}
\usepackage[T1]{fontenc}
\usepackage[utf8]{inputenc}
\usepackage{lmodern}
\usepackage{amsmath,amssymb,amsthm,mathtools}
\usepackage[margin=29mm]{geometry}
\usepackage{microtype}
\usepackage[numbers,sort&compress]{natbib}
\usepackage{xcolor}
\usepackage{etoolbox}
\usepackage[colorlinks=true,linkcolor=black,citecolor=black,urlcolor=black]{hyperref}
\makeatletter
\renewcommand*{\l@subsection}{\@tocline{2}{0pt}{2.5pc}{5pc}{}}
\patchcmd{\@setauthors}{\MakeUppercase{\authors}}{\authors}{}%
  {\PackageError{spherical-author-layout}{Unable to patch author capitalization}{}}
\patchcmd{\@setauthors}{\centering\footnotesize}{\centering\normalsize}{}%
  {\PackageError{spherical-author-layout}{Unable to patch author font size}{}}
\makeatother
\newtheorem{theorem}{Theorem}[section]
\newtheorem{proposition}[theorem]{Proposition}
\newtheorem{lemma}[theorem]{Lemma}
\newtheorem{corollary}[theorem]{Corollary}
\theoremstyle{definition}
\newtheorem{definition}[theorem]{Definition}
\newtheorem{remark}[theorem]{Remark}
\newtheorem{example}[theorem]{Example}
\numberwithin{equation}{section}
\newcommand{\R}{\mathbb R}
\newcommand{\Mloc}{\mathcal M_{\mathrm{loc}}}
\newcommand{\dd}{\,\mathrm d}

\newcommand{\Kplus}{K_+}
\DeclareMathOperator{\Argmin}{Argmin}

\DeclareMathOperator{\sgn}{sgn}
\allowdisplaybreaks[2]
\title[Radial pressureless Euler flow]{Radial pressureless Euler flow: structure and stability of absorption at the origin}
\author[Gaowei Cao]{Gaowei Cao\textsuperscript{*}}
\thanks{\textsuperscript{*}Gaowei Cao: Hetao Institute of Mathematics and Interdisciplinary Sciences, Shenzhen 518038, China. E-mail: \href{mailto:caogaowei@himis-sz.cn}{\nolinkurl{caogaowei@himis-sz.cn}}.}
\author[Feimin Huang]{Feimin Huang\textsuperscript{\textdagger}}
\thanks{\textsuperscript{\textdagger}Feimin Huang: Corresponding author. Academy of Mathematics and Systems Science, Chinese Academy of Sciences, Beijing 100190, China; School of Mathematical Sciences, University of Chinese Academy of Sciences, Beijing 100049, China. E-mail: \href{mailto:fhuang@amt.ac.cn}{\nolinkurl{fhuang@amt.ac.cn}}.}
\author[Guirong Tang]{Guirong Tang\textsuperscript{\textdaggerdbl}}
\thanks{\textsuperscript{\textdaggerdbl}Guirong Tang: School of Mathematical Sciences, Capital Normal University, Beijing 100048, China. E-mail: \href{mailto:tangguirong@amss.ac.cn}{\nolinkurl{tangguirong@amss.ac.cn}}.}
\makeatletter
\let\addresses\@empty
\makeatother
\subjclass[2020]{35L65, 35D30, 35A01, 35A02, 49J52}
\keywords{Pressureless Euler equations, spherical symmetry, entropy solutions, absorption at the origin, Stieltjes boundary impulse measure, stability of absorbed mass}
\date{}

\begin{document}
\begin{abstract}
We study absorption at the origin in radial pressureless Euler flow in $\R^n$, $n\geq2$.
For rotation-invariant nonnegative locally finite initial mass and bounded radial initial velocity vanishing at the origin, we prove global existence and uniqueness in a full-space radial entropy class with a one-sided Lipschitz bound and an initial kinetic-energy trace.
Absorption is permanent, including retention of any initial origin atom.

For finite positive mass with finite second moment, a Lagrangian projection representation characterizes the absorbed mass and its left limits at positive times through extreme minimizers of a free-position primitive.
We derive a radial scalar momentum balance with a nonnegative Stieltjes boundary impulse measure.
The minimum value of the same primitive yields a weighted Stieltjes identity equating accumulated time-weighted impulse with the integral of initial radii over absorbed labels, valid also for jumps and singular continuous absorption; the boundary impulse measure and the absorption measure are mutually absolutely continuous on $(0,\infty)$.

For such finite-mass data with common total mass, strong $L^2$ convergence of initial quantiles and Lagrangian velocities, together with convergence of initial origin masses, yields convergence of absorbed mass at continuity times of the limiting absorption and local weak convergence of absorption measures on $(0,\infty)$.
Finally, even with a fixed smooth initial radial probability density, every probability measure compactly supported in $(0,\infty)$ can be realized as the absorption measure by choosing a bounded radial initial velocity; in particular, purely singular continuous absorption occurs.
\end{abstract}
\maketitle
\vspace{-18pt}
\setcounter{tocdepth}{2}
\begingroup
\makeatletter
\patchcmd{\@starttoc}{\addvspace{32\p@\@plus14\p@}}{\addvspace{14\p@\@plus14\p@}}{}%
  {\PackageError{spherical-contents-layout}{Unable to adjust contents spacing}{}}
\makeatother
\fontsize{10}{11.5}\selectfont
\tableofcontents
\endgroup
\clearpage
\section{Introduction and main results}\label{sec:introduction}

We consider the pressureless Euler system
\begin{equation}\label{eq:euler}
 \partial_t\rho+\operatorname{div}(\rho\mathbf u)=0,\qquad
 \partial_t(\rho\mathbf u)+\operatorname{div}
       (\rho\mathbf u\otimes\mathbf u)=0
 \quad\hbox{in }(0,\infty)\times\R^n,
\end{equation}
for rotation-invariant, nonnegative locally finite mass in dimension $n\ge2$,
with radial velocity and no tangential component. In the sticky evolution
of finitely many spherical shells, a collision at positive radius
preserves the total radial momentum of the colliding shells and gives the merged shell their
mass-weighted mean velocity. A shell reaching the origin instead stops and remains there,
preserving mass but contributing an impulse to the radial scalar momentum balance.
This finite-shell description motivates our study of absorption for general radial mass measures.
We characterize the amount and timing of absorption, the associated radial scalar momentum
balance, and the stability of the absorbed mass under perturbations of the initial data.

To treat spherical shells, mass at the origin, and singular continuous radial mass
in a single representation, we use the physical radial mass measure
\begin{equation}\label{eq:intro-radial}
 \mu_t=|\cdot|_\#\rho_t,\qquad
 \rho_t=\int_{[0,\infty)}\sigma_r\,\mu_t(\dd r),\qquad
 \sigma_0=\delta_{\mathbf0}.
\end{equation}
Here $\sigma_r$ is the uniform probability measure on the sphere of radius
$r>0$. In particular, $\mu_t(\{0\})=\rho_t(\{\mathbf0\})$.
The origin is a boundary only for the radial coordinate, not for the full-space system.
We therefore distinguish the radial scalar momentum from the full-space vector momentum
when formulating the balance at the origin.

At the origin, angular averaging gives mass tests $\phi$ with $\phi_r(t,0)=0$
and radial momentum tests $\psi$ with $\psi(t,0)=0$.
These distinct conditions enter the radial weak formulation; the initial kinetic-energy trace
is imposed separately in the entropy class of Definition~\ref{def:entropy}.
Both the weak formulation and this trace must account for a possible origin atom.

The absorbed mass need not vary absolutely continuously in time: it may jump or increase
singular continuously. These features require a description of absorption and its radial
momentum contribution that accommodates the different types of temporal behavior.
We also study the conditions under which the absorbed mass is stable.

The variational description of one-dimensional pressureless flow goes
back to E, Rykov and Sinai \cite{ERS1996}. Brenier and Grenier
\cite{BG1998} relate sticky-particle dynamics to a scalar conservation law for
cumulative mass. Moutsinga \cite{Moutsinga2008} gives a convex-hull
construction and describes the evolution of cluster labels under a
negative-jump assumption on the initial velocity.
For an initial probability measure with finite second moment and square-integrable
initial velocity, Carrillo and Galtung
\cite{CG2026} identify scalar entropy solutions with Lagrangian gradient-flow
solutions. Their Oleinik E-condition on the cumulative flux is distinct
from the one-sided Lipschitz (OSL) bound \eqref{eq:radial-osl} on velocity used here.

Two parts of the one-dimensional measure-valued theory play distinct roles in our analysis.
Huang and Wang \cite{HW2001} prove Eulerian
existence and uniqueness for locally finite mass and bounded initial
velocity, with a velocity OSL bound and an initial
kinetic-energy trace. For finite mass with finite second moment and bounded initial velocity,
Natile and Savar\'e
\cite{NS2009} provide a whole-line Lagrangian projection representation in which
velocity is the conditional expectation of the initial velocity given the current position.
We use the Eulerian theory to establish radial existence and uniqueness after verifying
the reflection correspondence. Under the additional mass and moment assumptions,
we derive a radial projection representation of the same entropy solution and use it
to identify the absorbed labels.

Reflection and absorption already occur in the sticky-boundary theory
of Tudorascu \cite{Tudorascu2023}; Hynd and Tudorascu \cite{HT2024}
describe stopped Lagrangian maps and their long-time behavior.
For radial shadow waves, Nedeljkov, Neumann, Oberguggenberger and Sahoo
\cite{NNOS2017} study origin mass growth, shell arrival and radial
momentum exchange through boundary traces. The half-line inflow problem
of Neumann, Oberguggenberger, Sahoo and Sen \cite{NOSS2022} also has an
absorbing branch whose momentum formula subtracts the initial momentum
of the absorbed labels. The independently prescribed boundary inflow
in that problem is absent here. These works provide antecedents for reflection,
permanent absorption and the total momentum identity.
Our focus is the relation between the time-dependent absorbed mass and the Stieltjes boundary impulse measure,
including its atomic and atomless parts.

Existence and uniqueness depend on the chosen solution class. In general
multidimensional sticky dynamics, Bressan and Nguyen \cite{BN2014}
exhibit smooth compactly supported initial densities and continuous
velocities for which their Lagrangian solution class has no local
solution or has more than one. Here the full-space and half-line weak and entropy classes are
defined independently in Definitions~\ref{def:full-entropy} and~\ref{def:entropy}.
Proposition~\ref{prop:radial} proves their equivalence.
The full-space entropy class requires, in addition to the weak formulation
and its time-representative requirements, the vector OSL bound
\eqref{eq:full-osl} outside a common null set of positive times and the
initial kinetic-energy trace \eqref{eq:full-initial-energy-trace}.
No particle or projection representation is prescribed.
The following theorem establishes existence and uniqueness in this class
without a total-mass or moment restriction.

\begin{theorem}[Existence and uniqueness of radial entropy solutions]
\label{thm:main}
Let $n\geq2$, and let $\rho_0$ be a rotation-invariant nonnegative
locally finite Radon measure on $\R^n$. Set $\mu_0=|\cdot|_\#\rho_0$.
Let $u_0\in L^\infty([0,\infty),\mu_0)$ with $u_0(0)=0$, and prescribe
the radial initial velocity
\[
 \mathbf u_0(x)=u_0(|x|)\frac{x}{|x|}\quad(x\ne\mathbf0),
 \qquad \mathbf u_0(\mathbf0)=\mathbf0.
\]
Then \eqref{eq:euler}, with initial data $(\rho_0,\mathbf u_0)$, admits
a global radial entropy solution $(\rho,\mathbf u)$ in the sense of
Definition~\ref{def:full-entropy}. Within this class, for every $t\geq0$,
the mass $\rho_t$ and vector momentum $\mathbf u(t,\cdot)\rho_t$ are
uniquely determined; consequently, $\mathbf u(t,\cdot)$ is unique
$\rho_t$-almost everywhere and satisfies
$\|\mathbf u(t,\cdot)\|_{L^\infty(\rho_t)}
\leq\|u_0\|_{L^\infty(\mu_0)}$.
No uniqueness of velocity representatives in vacuum is asserted.
\end{theorem}

The proof passes to the physical radial mass and then, by reflection,
to the symmetric whole-line problem. The reflection is
\begin{equation}\label{eq:intro-reflect}
 \begin{gathered}
 \lambda_t=\tfrac12\bigl((j_+)_\#\mu_t+(j_-)_\#\mu_t\bigr),
 \qquad j_\pm(r)=\pm r,\\
 v(t,x)=\sgn(x)u(t,|x|),\qquad v(t,0)=0.
 \end{gathered}
\end{equation}
Proposition~\ref{prop:radial} identifies the full-space and half-line
weak and entropy formulations. Theorem~\ref{thm:reflection} verifies the reflection for the
weak equations, the origin atom, the OSL bound across zero and the initial
traces, allowing us to apply the Huang--Wang theorem.

In what follows, $(\mu,u)$ denotes the radial representation of the
full-space solution $(\rho,\mathbf u)$: $\mu$ is its physical radial mass
and $u$ its scalar radial velocity, related by \eqref{eq:intro-radial}
and \eqref{eq:radial-velocity}. Initial data $(\mu_0,u_0)$ are understood
in the same radial coordinates.

\begingroup
Permanent absorption at the origin holds in the locally finite-mass class of
Theorem~\ref{thm:main}. Proposition~\ref{prop:sticking} gives
\[
 \mu_t(\{0\})\geq\mu_s(\{0\})\qquad(0\leq s\leq t).
\]
The proof uses the radial mass balance and the OSL bound to compare positive
times, and the initial kinetic-energy trace to retain any initial origin atom.
\endgroup

For finite nonzero total mass with finite second moment,
the absorbed mass is the length of the zero level set of the
projected position map in mass coordinates. The next theorem identifies
this mass as the largest minimizer of the free-position primitive
and gives the radial scalar momentum balance.

For a continuous function $A$ on a nonempty compact interval $I$, write
\[
 \Argmin_{a\in I}A(a)
 :=\{a\in I:A(a)=\min_{b\in I}A(b)\}.
\]
The expressions $\max\Argmin$ and $\min\Argmin$ denote the largest
and smallest elements of this minimizer set, respectively.

For the next theorem, assume in addition to the hypotheses of
Theorem~\ref{thm:main} that $\mu_0$ has finite nonzero total mass
$m=\mu_0([0,\infty))$ and finite second moment. Define the radial
quantile and initial Lagrangian velocity by
\[
 \begin{gathered}
 R_0(a)=\inf\{r\geq0:\mu_0([0,r])>a\}\quad(0<a<m),
 \qquad U_0=u_0\circ R_0,\\
 L=\|U_0\|_{L^\infty(0,m)},\qquad m_0=\mu_0(\{0\}),
 \end{gathered}
\]
and set
\[
 \begin{gathered}
 \Kplus=\{R\in L^2(0,m):R\geq0\text{ and }R
                  \text{ has a nondecreasing representative}\},\\
 B(a)=\int_0^aR_0(b)\,\dd b,\qquad
 C(a)=\int_0^aU_0(b)\,\dd b,\qquad A_t=B+tC.
 \end{gathered}
\]
The map $R_0+tU_0$ is the free-position map, and $A_t$ is its
primitive with respect to the mass label.
Here $P_{\Kplus}$ denotes the metric position projection in $L^2(0,m)$ onto
$\Kplus$.

\begin{theorem}[Absorption structure and radial momentum balance]
\label{thm:main-absorption}
Under the above assumptions and notation, the radial representation
$(\mu,u)$ of the solution in Theorem~\ref{thm:main} has the following
properties.
\begin{enumerate}
\item[\textup{(i)}] For every $t\geq0$,
\begin{equation}\label{eq:intro-project}
 R_t=P_{\Kplus}(R_0+tU_0),\qquad \mu_t=(R_t)_\#\dd a.
\end{equation}
Its absorbed mass $M(t)=\mu_t(\{0\})$ is nondecreasing and right
continuous on $[0,\infty)$, with $M(0)=m_0$, and
\begin{equation}\label{eq:intro-absorption}
 \begin{aligned}
 M(t)&=\max\Argmin_{a\in[m_0,m]}A_t(a)&& (t\geq0),\\
 M(t-)&=\min\Argmin_{a\in[m_0,m]}A_t(a)&& (t>0).
 \end{aligned}
\end{equation}
Thus the absorption jump at a positive time is the difference of these
two extreme minimizers.
\item[\textup{(ii)}] The total radial momentum
$P(t)=\int_{[0,\infty)}u(t,r)\,\mu_t(\dd r)$ is nondecreasing and
right continuous on $[0,\infty)$, and satisfies
\begin{equation}\label{eq:intro-total-momentum}
 P(t)=C(m)-C(M(t))\quad(t\geq0).
\end{equation}
Its Stieltjes measure $\mathcal I=\dd P$ on $(0,\infty)$ satisfies
\begin{equation}\label{eq:intro-impulse}
 0\leq\mathcal I\leq L\,\dd M,
 \qquad t\mathcal I=\dd(B\circ M)\quad\text{on }(0,\infty).
\end{equation}
For every $\psi\in C_c^1([0,\infty)\times[0,\infty))$, the radial
scalar momentum satisfies
\begingroup
\begin{equation}\label{eq:intro-boundary-momentum-weak}
 \begin{split}
 &\int_0^\infty\!\int_{[0,\infty)}
       (u\psi_t+u^2\psi_r)\,\dd\mu_t\,\dd t
  +\int_{[0,\infty)}u_0(r)\psi(0,r)\,\dd\mu_0\\
 &\hspace{38mm}+\int_{(0,\infty)}\psi(t,0)\,\mathcal I(\dd t)=0.
 \end{split}
\end{equation}
\endgroup
 No vanishing condition is imposed on $\psi(t,0)$.
\end{enumerate}
\end{theorem}

The Stieltjes absorption measure on $(0,\infty)$ records newly
absorbed mass:
\[
 (\dd M)((s,t])=M(t)-M(s)\qquad(0\leq s<t).
\]
The initial origin mass is retained separately as $M(0)=m_0$.

At a jump, the interval between the two minimizers in
\eqref{eq:intro-absorption} consists, up to endpoints, of the newly absorbed
labels. Their contribution to the jump of $P$ is the negative of their
total initial momentum.

In {\eqref{eq:intro-boundary-momentum-weak}},
$\mathcal I$ supplies its boundary impulse term for tests that may be nonzero at the
origin. The full-space vector momentum satisfies \eqref{eq:euler}
without a source at the origin.

\begingroup
Corollary~\ref{cor:impulse-types} further gives separate formulas for the
atomic and atomless parts of $\mathcal I$ and shows that $\mathcal I$
and $\dd M$ are mutually absolutely continuous on $(0,\infty)$.
\endgroup

\begingroup
On $[m_0,m]$, every minimizer at an earlier time is no larger than any
minimizer at a later time. Together with the uniform convergence
$A_s\to A_t$ as $s\to t$, this gives the left-limit formula in
\eqref{eq:intro-absorption}; see the proof of
Theorem~\ref{thm:absorption}.
\endgroup

\begingroup
The velocity projection \eqref{eq:radial-actual-velocity} gives the total
momentum formula \eqref{eq:intro-total-momentum}. For the minimum-value
function $h$ in \eqref{eq:minimum-value-function},
\[
 h(t)=\min_{a\in[m_0,m]}A_t(a)=B(M(t))+tC(M(t)).
\]
Comparison of its minimum values yields
$\dd h=C(M(t))\dd t$. The BV product rule then gives
$\dd(B\circ M)=t\dd P=t\mathcal I$ on $(0,\infty)$.
This derivation applies to jumps and to both absolutely continuous
and singular continuous absorption.
\endgroup

Together, Theorems~\ref{thm:projection}, \ref{thm:absorption} and
\ref{thm:boundary-momentum} prove Theorem~\ref{thm:main-absorption}.

Section~\ref{sec:absorption} also gives onset, completion,
asymptotic and local-rate criteria.

Convergence of projected positions does not by itself control the
length of the zero level set. With unit total mass and zero
initial velocity, the stationary quantiles $R_t^k\equiv1/k$ converge in
$L^2(0,1)$ to $R_t\equiv0$, while the absorbed masses remain zero and one,
respectively. The next proposition combines strong $L^2$ convergence of the initial
quantiles and Lagrangian velocities with convergence of the initial origin masses.

\begin{proposition}[Stability of absorbed mass]\label{thm:main-stability}
Let $(\mu_0^k,u_0^k)$ and $(\mu_0,u_0)$ each satisfy the assumptions of
Theorem~\ref{thm:main-absorption}, with common total mass $m$.
For their initial quantiles and Lagrangian velocities, assume
\[
 R_0^k\longrightarrow R_0,\qquad U_0^k\longrightarrow U_0
       \quad\text{in }L^2(0,m),\qquad
 m_0^k:=\mu_0^k(\{0\})\longrightarrow m_0:=\mu_0(\{0\}).
\]
Then the absorbed masses of the corresponding solutions satisfy
\[
 M(t-)\leq\liminf_{k\to\infty}M_k(t)
 \leq\limsup_{k\to\infty}M_k(t)\leq M(t)\qquad(t>0).
\]
In particular, $M_k(t)\to M(t)$ at every continuity time of $M$,
including $t=0$, and $\dd M_k\rightharpoonup\dd M$ locally weakly on
$(0,\infty)$, against every test in $C_c((0,\infty))$.
\end{proposition}

If the limiting initial mass has no origin atom, $m_0=0$, the
condition $m_0^k\to0$ follows from $R_0^k\to R_0$ in $L^2(0,m)$;
see Section~\ref{subsec:absorption-stability}.

The proof in Section~\ref{subsec:absorption-stability} uses
the strong $L^2(0,m)$ convergence of the initial quantiles and
Lagrangian velocities to obtain uniform convergence of the free-position
primitives on $[0,T]\times[0,m]$ for each finite $T>0$; see
\eqref{eq:primitive-uniform-convergence}. Together with $m_0^k\to m_0$,
this implies that, for each fixed $t>0$, every cluster point of the
minimizers $M_k(t)$ on the moving intervals $[m_0^k,m]$ belongs to
$\Argmin_{a\in[m_0,m]}A_t(a)$. The extreme minimizers in
\eqref{eq:intro-absorption} give the two-sided bound at a jump and convergence
at positive continuity times. Stieltjes integration by parts, this
convergence and the bound $0\leq M_k\leq m$ then give local weak
convergence of the absorption measures on $(0,\infty)$.

Proposition~\ref{prop:absorption-realization} shows
that, for any fixed initial radial mass with a smooth probability
density compactly supported in $(1,2)$, every probability measure
compactly supported in $(0,\infty)$ can be realized as
the absorption measure of the corresponding entropy solution by choosing
a bounded Borel radial initial velocity; in particular, singular
continuous absorption does occur.
This supplies examples for which the singular continuous
part of the Stieltjes formulation is essential.

Section~\ref{sec:framework} introduces physical radial mass, defines
full-space and half-line weak and entropy solutions independently, and
proves the equivalence of the two formulations. Section~\ref{sec:reflection}
proves reflection equivalence, global existence and uniqueness,
using the one-dimensional Eulerian theory and the required adaptations
of primitives and velocity representatives.
Section~\ref{sec:absorption} establishes permanent absorption and
develops the quantitative absorption structure, boundary momentum and stability,
using the finite-mass Lagrangian representation for the quantitative results.

\section{Radial measures and entropy framework}\label{sec:framework}

Throughout the paper, $n\geq2$. We consider rotation-invariant nonnegative
Radon measures on $\mathbb R^n$ and radial velocities without a tangential
component. The latter assumption is explicit: in dimension two, rotation
equivariance alone also permits a tangential velocity.

\subsection[Physical radial mass and velocity]{Physical radial mass and velocity}

The physical radial mass measure treats atomless mass, atoms at positive radii, and the
origin atom within a single representation.
Let $\sigma_r$ be the uniform probability measure on the sphere of radius
$r>0$, and set $\sigma_0=\delta_{\mathbf0}$. If $\rho$ is rotation invariant,
its physical radial mass measure is
\begin{equation}\label{eq:physical-radial-mass}
 \mu=|\cdot|_\#\rho,\qquad
 \rho=\int_{[0,\infty)}\sigma_r\,\mu(\mathrm dr).
\end{equation}
Both measures are locally finite, and $\rho(\{\mathbf0\})=\mu(\{0\})$.
Indeed, averaging a compactly supported continuous function over rotations
gives its spherical average; integration against $\rho$ then proves
\eqref{eq:physical-radial-mass}. This also proves uniqueness of the
representation. No decomposition of the singular radial mass into atoms is required. For example, an atomless Cantor measure on the radius
interval $[1,2]$ defines a singular continuous spatial mass through
\eqref{eq:physical-radial-mass}; it assigns zero mass to every individual
sphere. Such data are included in the same framework.

For a scalar radial velocity $u$, we write
\begin{equation}\label{eq:radial-velocity}
 \mathbf u(t,x)=u(t,|x|)\frac{x}{|x|}\quad(x\ne\mathbf0),
 \qquad \mathbf u(t,\mathbf0)=\mathbf0,
 \qquad u(t,0)=0.
\end{equation}
The value at the origin is part of the velocity convention and is required
on an origin atom. In particular, such an atom has zero momentum and zero
kinetic energy. If $\rho$ has a density $f(|x|)$ away from the origin, then
$\mu(\mathrm dr)=|\mathbb S^{n-1}|r^{n-1}f(r)\,\mathrm dr$ there.
We do not require the measure obtained by dividing $\mu|_{(0,\infty)}$ by
$r^{n-1}$ to have finite mass near $r=0$.

\subsection[Full-space and radial weak and entropy solutions]{Full-space and radial weak and entropy solutions}

\begingroup
We first define the full-space radial class directly in terms of spatial
mass and vector momentum, and then give the half-line definition.
In the full-space formulation, local weak convergence is tested against
$C_c(\mathbb R^n)$, or its vector-valued counterpart for momentum.
In the full-space weak equations, smooth scalar tests $\Phi$ and vector
tests $\Psi$ on $[0,\infty)\times\mathbb R^n$ mean restrictions of
compactly supported smooth functions on $\mathbb R\times\mathbb R^n$;
they may be nonzero at $t=0$, and no radial symmetry or condition at
the spatial origin is imposed on them.

Fix full-space initial data
\begin{equation}\label{eq:full-data}
 \rho_0\in\mathcal M^+_{\mathrm{loc}}(\mathbb R^n),\qquad
 \mathbf u_0\in L^\infty(\rho_0;\mathbb R^n),\qquad
 \mathbf u_0(\mathbf0)=\mathbf0,
\end{equation}
where $\rho_0$ is rotation invariant and $\mathbf u_0$ has a bounded
Borel representative that is radial without a tangential component.
Explicitly, the chosen representative satisfies
\[
 \mathbf u_0(Qx)=Q\mathbf u_0(x)\quad(Q\in SO(n),\ x\in\mathbb R^n),
 \qquad \mathbf u_0(x)\in\mathbb R x\quad(x\ne\mathbf0).
\]
The origin condition is a condition on the velocity at an origin atom,
and fixes its value there otherwise; it imposes no continuity at the origin.

\begin{definition}[Full-space radial weak and entropy solutions]
\label{def:full-entropy}
A full-space radial weak solution with data \eqref{eq:full-data}
is a pair $(\rho,\mathbf u)$, where each $\rho_t$, $t\geq0$, is a
nonnegative locally finite rotation-invariant Radon measure on
$\mathbb R^n$, and
$\mathbf u:[0,\infty)\times\mathbb R^n\to\mathbb R^n$ is jointly Borel.
The chosen velocity representative has the form
\eqref{eq:radial-velocity} at every point and satisfies
$\sup_{t\geq0,\,x\in\mathbb R^n}|\mathbf u(t,x)|\leq U<\infty$.
The mass $\rho_t$ and vector momentum $\mathbf u(t,\cdot)\rho_t$ are
locally weakly continuous on $[0,\infty)$, with respective initial
values $\rho_0$ and $\mathbf u_0\rho_0$.
The weak equations are
\begingroup
\begin{align}
 &\int_0^\infty\!\int_{\mathbb R^n}
  (\Phi_t+\mathbf u\cdot\nabla\Phi)\,\mathrm d\rho_t\,\mathrm dt
  +\int_{\mathbb R^n}\Phi(0,x)\,\mathrm d\rho_0=0,
       \label{eq:full-mass-weak}\\
 &\int_0^\infty\!\int_{\mathbb R^n}
  \bigl(\mathbf u\cdot\Psi_t+
     (\mathbf u\otimes\mathbf u):D\Psi\bigr)
     \,\mathrm d\rho_t\,\mathrm dt
  +\int_{\mathbb R^n}\mathbf u_0\cdot\Psi(0,x)
       \,\mathrm d\rho_0=0,\label{eq:full-momentum-weak}
\end{align}
\endgroup
for all smooth scalar tests $\Phi$ and vector tests $\Psi$ on
$[0,\infty)\times\mathbb R^n$ in the sense specified above.

It is a full-space radial entropy solution if, in addition, there is a
Lebesgue-null set $N\subset(0,\infty)$, independent of $x$ and $y$, such that
the chosen representative satisfies
\begin{equation}\label{eq:full-osl}
 (\mathbf u(t,x)-\mathbf u(t,y))\cdot(x-y)
 \leq\frac{|x-y|^2}{t}
 \quad(x,y\in\mathbb R^n)
\end{equation}
for every $t\in(0,\infty)\setminus N$, and
\begin{equation}\label{eq:full-initial-energy-trace}
 |\mathbf u(t,\cdot)|^2\rho_t\rightharpoonup
 |\mathbf u_0|^2\rho_0\qquad(t\downarrow0).
\end{equation}
\end{definition}

\par\endgroup

For the half-line formulation, fix initial data
\begin{equation}\label{eq:radial-data}
 \mu_0\in\mathcal M^+_{\mathrm{loc}}([0,\infty)),\qquad
 u_0\in L^\infty(\mu_0),\qquad u_0(0)=0.
\end{equation}
The last equality imposes a condition when $\mu_0(\{0\})>0$ and fixes
the velocity value at the origin otherwise. Put $q_0=u_0\mu_0$ and
$e_0=u_0^2\mu_0$. We use twice the kinetic energy in the notation $e_0$.
All weak convergences of locally finite measures below are tested against
compactly supported continuous functions.

Angular averaging gives different boundary compatibility conditions for
mass and momentum tests. We retain these conditions in the time
representatives and impose the initial kinetic-energy trace separately.

\begin{definition}[Half-line radial weak and entropy solutions]
\label{def:entropy}
A radial weak solution with data \eqref{eq:radial-data} consists of a
locally weakly continuous curve
$\mu:[0,\infty)\to\mathcal M^+_{\mathrm{loc}}([0,\infty))$ with
$\mu|_{t=0}=\mu_0$, and a jointly Borel velocity
$u:[0,\infty)\times[0,\infty)\to\mathbb R$ whose chosen representative
satisfies $u(t,0)=0$ for all $t\geq0$ and
$\sup_{t\geq0,\,r\geq0}|u(t,r)|\leq U<\infty$.
Write $q_t=u(t,\cdot)\mu_t$ for the radial scalar momentum measure and
set $e_t=u(t,\cdot)^2\mu_t$.
The radial momentum $q_t$ is continuous on $[0,\infty)$ against every
$\psi\in C_c([0,\infty))$ with $\psi(0)=0$, with initial value $q_0$.
The equations are
\begin{align}
 &\int_0^\infty\!\int_{[0,\infty)}
     (\phi_t+u\phi_r)\,\mathrm d\mu_t\,\mathrm dt
   +\int_{[0,\infty)}\phi(0,r)\,\mathrm d\mu_0=0,
       \label{eq:radial-mass-weak}\\
 &\int_0^\infty\!\int_{[0,\infty)}
     (u\psi_t+u^2\psi_r)\,\mathrm d\mu_t\,\mathrm dt
   +\int_{[0,\infty)}u_0(r)\psi(0,r)\,\mathrm d\mu_0=0,
       \label{eq:radial-momentum-weak}
\end{align}
for all $\phi,\psi\in C_c^1([0,\infty)\times[0,\infty))$ satisfying
$\phi_r(t,0)=0$ and $\psi(t,0)=0$, respectively.
It is an entropy solution if, in addition, there is a Lebesgue-null set
$N\subset(0,\infty)$, independent of $r_1$ and $r_2$, such that the
chosen representative satisfies
\begin{equation}\label{eq:radial-osl}
 u(t,r_2)-u(t,r_1)\leq\frac{r_2-r_1}{t}
 \quad(0\leq r_1<r_2)
\end{equation}
for every $t\in(0,\infty)\setminus N$, and
\begin{equation}\label{eq:initial-energy-trace}
 e_t\rightharpoonup e_0\qquad(t\downarrow0).
\end{equation}
\end{definition}

The continuity requirement for $q_t$ uses tests vanishing at the origin
because the radial scalar momentum of an incoming shell changes when
the shell reaches the origin, whereas its full-space vector momentum has a
continuous weak representative. The restriction on momentum tests in
\eqref{eq:radial-momentum-weak} records this distinction. It is not a
statement of scalar momentum conservation on a half-line without a
boundary interaction.

At a positive-mass radius $r>0$, the velocity is fixed by
$u(t,r)=q_t(\{r\})/\mu_t(\{r\})$; changing a value at such an atom changes
the solution.

\emph{Initial radial scalar momentum trace.}
The initial kinetic-energy trace restores the initial trace of the radial scalar momentum
against all compactly supported continuous tests. To control the
contribution near zero, let $\chi_\delta$ be a cutoff equal to one on
$[0,\delta]$ and supported in $[0,2\delta]$. Uniform local mass bounds
and Cauchy--Schwarz give
\[
 \int\chi_\delta |u_t|\,\mathrm d\mu_t
 \leq\left(\int\chi_\delta\,\mathrm d\mu_t\right)^{1/2}
      \left(\int\chi_\delta\,\mathrm de_t\right)^{1/2}.
\]
First letting $t\downarrow0$ and then $\delta\downarrow0$ makes the
right-hand side vanish, since $e_0(\{0\})=0$. Away from zero the
required convergence follows from the momentum representative in
Definition~\ref{def:entropy}. Splitting a general test with
$\chi_\delta$ proves $q_t\rightharpoonup q_0$ as $t\downarrow0$.
This conclusion at the initial time does not impose scalar momentum
continuity at later absorption times.

\emph{Extension of mass tests.}
For the mass identity, the derivative compatibility at the origin can be
removed. In particular, it can be tested with every compactly supported $C^1$
function, without the condition $\phi_r(t,0)=0$. To see this, modify a
test on $0\leq r\leq2\varepsilon$ so that its radial derivative vanishes at
zero, while keeping its value and time derivative uniformly convergent and
its radial derivative uniformly bounded. The only additional error is
bounded by a constant times
$\int_0^T\int_{(0,2\varepsilon)}|u|\,\mathrm d\mu_t\,\mathrm dt$.
This tends to zero by dominated convergence. There is no contribution from
the atom at zero because $u(t,0)=0$.

\subsection[Equivalence of the full-space and radial formulations]{Equivalence of the full-space and radial formulations}

\begingroup
The two independently defined classes agree under radial reconstruction,
including their time representatives and entropy conditions.

\begin{proposition}[Full-space and radial weak and entropy formulations]
\label{prop:radial}
Let $(\rho_0,\mathbf u_0)$ and $(\mu_0,u_0)$ satisfy
\eqref{eq:full-data} and \eqref{eq:radial-data}, respectively, with the
representative conventions specified there, and be related by
\eqref{eq:physical-radial-mass}--\eqref{eq:radial-velocity} at the initial
time. These transformations identify the full-space radial weak
(respectively, entropy) solutions of Definition~\ref{def:full-entropy}
with the radial weak (respectively, entropy) solutions of
Definition~\ref{def:entropy}, at the level of mass and momentum measures
at every time.
\end{proposition}
\par\endgroup

\begin{proof}
\begingroup
\emph{Reconstruction and local bounds.}
The spherical representation \eqref{eq:physical-radial-mass} is unique.
For a chosen full-space radial representative,
fix a unit vector $e_1$ and set
\[
 u(t,r)=\mathbf u(t,re_1)\cdot e_1\quad(r>0),\qquad u(t,0)=0.
\]
Rotation equivariance and the absence of a tangential component make this
scalar value independent of the unit direction used. They also give
\eqref{eq:radial-velocity} at every point. This construction and the
reverse reconstruction preserve joint Borel measurability, the same bound
$U$, and the origin convention. The same observations apply to the
initial data, after choosing a bounded Borel representative of $u_0$.

Scalar spherical averaging and the continuous lift $f(|x|)$ identify the
mass pairings at every time, so local weak continuity of mass is preserved
in either direction. If $T,R<\infty$, a continuous compactly supported
radial cutoff that equals one on $[0,R]$ consequently gives
\[
 \sup_{0\leq t\leq T}\mu_t([0,R])
 =\sup_{0\leq t\leq T}\rho_t(\{x:|x|\leq R\})<\infty.
\]
Together with the velocity bound, this controls the local momentum and
quadratic flux integrals. No bound on total mass is used.

\emph{Weak equations and initial terms.}
\par\endgroup

Let $\omega$ denote the uniform probability measure on $\mathbb S^{n-1}$.
For full-space tests $\Phi$ and $\Psi$, set
\[
 \phi(t,r)=\int_{\mathbb S^{n-1}}\Phi(t,r\theta)\,\mathrm d\omega(\theta),
 \qquad
 \psi(t,r)=\int_{\mathbb S^{n-1}}
       \Psi(t,r\theta)\cdot\theta\,\mathrm d\omega(\theta).
\]
Differentiation under the angular integral gives
\[
 \phi_r=\int\nabla\Phi(t,r\theta)\cdot\theta\,\mathrm d\omega,
 \qquad
 \psi_r=\int\theta\cdot D\Psi(t,r\theta)\theta\,\mathrm d\omega.
\]
Since $\int\theta\,\mathrm d\omega=0$, these are admissible radial tests:
$\phi_r(t,0)=0$ and $\psi(t,0)=0$. The measure representation and the radial
velocity formula yield
\begin{equation}\label{eq:angular-weak-pairings}
\begin{aligned}
 \int_{\mathbb R^n}(\Phi_t+\mathbf u\cdot\nabla\Phi)\,\mathrm d\rho_t
 &=\int_{[0,\infty)}(\phi_t+u\phi_r)\,\mathrm d\mu_t,\\
 \int_{\mathbb R^n}\bigl(\mathbf u\cdot\Psi_t+
   (\mathbf u\otimes\mathbf u):D\Psi\bigr)\,\mathrm d\rho_t
 &=\int_{[0,\infty)}(u\psi_t+u^2\psi_r)\,\mathrm d\mu_t.
\end{aligned}
\end{equation}
The identities in \eqref{eq:angular-weak-pairings} include $r=0$:
the mass term there is
$\phi_t(t,0)\mu_t(\{0\})$, while both momentum terms vanish. The initial
terms satisfy
\begin{equation}\label{eq:angular-initial-pairings}
\begin{aligned}
 \int_{\mathbb R^n}\Phi(0,x)\,\mathrm d\rho_0
   &=\int_{[0,\infty)}\phi(0,r)\,\mathrm d\mu_0,\\
 \int_{\mathbb R^n}\mathbf u_0\cdot\Psi(0,x)\,\mathrm d\rho_0
   &=\int_{[0,\infty)}u_0(r)\psi(0,r)\,\mathrm d\mu_0.
\end{aligned}
\end{equation}
Applying the two radial weak equations to $\phi$ and $\psi$ proves
\eqref{eq:full-mass-weak}--\eqref{eq:full-momentum-weak} for arbitrary
full-space tests.

For the converse, the compatible radial tests $\phi$ and $\psi$ have
lifts
\[
 \Phi(t,x)=\phi(t,|x|),\qquad
 \Psi(t,x)=\psi(t,|x|)\frac{x}{|x|}
 \quad(x\ne\mathbf0),\qquad
 \Psi(t,\mathbf0)=\mathbf0.
\]
\begingroup
The condition $\phi_r(t,0)=0$ makes the scalar lift $C^1$ at the origin,
with spatial derivative zero there. Its time derivative is
$\phi_t(t,|x|)$ and is continuous at the origin as well.
\par\endgroup

For the vector lift, write $r=|x|$ and $\theta=x/r$. At positive radii,
\[
 D\Psi(t,x)
 =\psi_r(t,r)\,\theta\otimes\theta
  +\frac{\psi(t,r)}r\bigl(I-\theta\otimes\theta\bigr)
 \longrightarrow\psi_r(t,0)I\qquad(r\downarrow0).
\]
\begingroup
Indeed,
$\psi(t,r)/r=\int_0^1\psi_r(t,ar)\,\mathrm da$ tends to
$\psi_r(t,0)$ locally uniformly in $t$. The time derivative of the
vector lift extends continuously by zero, since $\psi_t(t,0)=0$.
Both lifts are therefore $C^1$ jointly up to $t=0$ and at the spatial
origin; only $C^1$ regularity of the radial tests is used.

For completeness, each lift $F$, scalar or vector, extends to $t<0$ by
$3F(-t,x)-2F(-2t,x)$. Its value and first derivatives match at $t=0$.
The extension is compactly supported and $C^1$ on
$\mathbb R\times\mathbb R^n$. Mollifying it in all variables gives
smooth tests converging in $C^1$, with supports in one fixed compact set.
Restrict these tests to $t\geq0$. The velocity bound and the uniform
local mass bound control all spacetime integrals; local finiteness of
$\rho_0$ and boundedness of $\mathbf u_0$ control the initial terms.
The full-space weak equations thus hold for the $C^1$ lifts themselves.
Their angular averages are exactly $\phi$ and $\psi$.
Using \eqref{eq:angular-weak-pairings} and
\eqref{eq:angular-initial-pairings} recovers
\eqref{eq:radial-mass-weak}--\eqref{eq:radial-momentum-weak}.

\emph{Time representatives and initial kinetic-energy traces.}
For $f\in C_c(\mathbb R^n)$ and
$G\in C_c(\mathbb R^n;\mathbb R^n)$, their spherical averages
\[
 \overline f(r)=\int f(r\theta)\,\mathrm d\omega(\theta),\qquad
 \overline G(r)=\int G(r\theta)\cdot\theta\,\mathrm d\omega(\theta)
\]
are compactly supported and continuous on $[0,\infty)$, with
$\overline G(0)=0$. At every time, the mass and momentum pairings satisfy
\[
 \int f\,\mathrm d\rho_t=\int\overline f\,\mathrm d\mu_t,\qquad
 \int G\cdot\mathbf u(t,\cdot)\,\mathrm d\rho_t
     =\int\overline G\,\mathrm dq_t.
\]
Conversely, every continuous compactly supported radial mass test $h$
has the continuous lift $h(|x|)$ with spherical average $h$.
Every $k\in C_c([0,\infty))$ with $k(0)=0$ has the continuous vector lift
$k(|x|)x/|x|$, extended by zero at the origin, whose vector spherical
average is $k$. These pairings identify the two time-continuity
requirements and their initial values. Since $q_t(\{0\})=0$, the tests
vanishing at zero also determine $q_t$ completely as a locally finite
signed measure. Hence the correspondence identifies mass and momentum
at every time, not just at almost every time. This does not impose
continuity of $q_t$ against tests nonzero at the origin at positive times.
\par\endgroup

Finally,
\[
 |\mathbf u(t,\cdot)|^2\rho_t
   =\int_{[0,\infty)}\sigma_r\,e_t(\mathrm dr),
 \qquad
 |\mathbf u_0|^2\rho_0
   =\int_{[0,\infty)}\sigma_r\,e_0(\mathrm dr).
\]
\begingroup
Scalar spherical averaging and radial lifting identify the local weak
initial traces in \eqref{eq:full-initial-energy-trace} and
\eqref{eq:initial-energy-trace} in both directions. Neither a finite
total energy nor energy continuity at positive times is required.

\emph{Equivalence of the OSL conditions.}
Fix $t>0$ at which \eqref{eq:radial-osl} holds for every radius pair.
Taking one radius to be zero gives $u(t,r)\leq r/t$ for all $r>0$.
Suppress $t$ in the notation. If $x=r\theta$ and $y=s\eta$ with
$r\geq s>0$ and $\theta,\eta\in\mathbb S^{n-1}$, then
\[
\begin{aligned}
 (\mathbf u(x)-\mathbf u(y))\cdot(x-y)
 &=(u(r)-u(s))(r-s)
   +(1-\theta\cdot\eta)\bigl(su(r)+ru(s)\bigr)\\
 &\leq\frac{(r-s)^2+2rs(1-\theta\cdot\eta)}{t}
  =\frac{|x-y|^2}{t}.
\end{aligned}
\]
For $r>s$ the first term is controlled by the radial OSL condition;
for $r=s$ it is zero. The second term is controlled by
$su(r)+ru(s)\leq2rs/t$ and $1-\theta\cdot\eta\geq0$.
If exactly one point is the origin, the inequality is
$r u(r)\leq r^2/t$; if both are the origin, it is trivial.
Exchanging $x$ and $y$ covers the remaining order of the radii.
Thus \eqref{eq:full-osl} holds for all point pairs at this same time.

Conversely, suppose \eqref{eq:full-osl} holds for all point pairs at a
fixed $t>0$. Set $x=r_2\theta$ and $y=r_1\theta$ with
$0\leq r_1<r_2$; when $r_1=0$, use the prescribed origin values.
The resulting inequality is
\[
 (u(t,r_2)-u(t,r_1))(r_2-r_1)\leq\frac{(r_2-r_1)^2}{t}.
\]
Dividing by $r_2-r_1>0$ proves \eqref{eq:radial-osl}.
Both implications are pointwise at each fixed good time, so the same
exceptional null set works in either formulation. Together with the
weak and trace correspondences above, this proves the entropy equivalence.
\endgroup
\end{proof}

\begingroup
Solutions in either formulation are compared through their mass and
momentum measures at every time, which determines their velocities
almost everywhere with respect to mass. Velocity representatives in
vacuum are chosen to satisfy boundedness, joint Borel measurability,
and the corresponding OSL condition in the entropy class.
\par\endgroup

\section[Existence and uniqueness of radial entropy solutions]{Existence and uniqueness of radial entropy solutions}
\label{sec:reflection}\label{sec:wellposedness}

We prove Theorem~\ref{thm:main} by reflecting the half-line
formulation to the whole line and applying the one-dimensional entropy theory.

\subsection{Reflection equivalence}\label{subsec:reflection-equivalence}

We use the following whole-line Euler entropy class with initial data
$(\lambda_0,v_0)$. The curves
$\lambda_t\geq0$ and $p_t=v_t\lambda_t$ take values in the space of locally
finite measures and are locally weakly continuous, $v$ is jointly Borel and bounded, and
\begin{align}
 &\int_0^\infty\!\int_{\mathbb R}
    (\zeta_t+v\zeta_x)\,\mathrm d\lambda_t\,\mathrm dt
    +\int_{\mathbb R}\zeta(0,x)\,\mathrm d\lambda_0=0,
       \label{eq:line-mass}\\
 &\int_0^\infty\!\int_{\mathbb R}
    (v\eta_t+v^2\eta_x)\,\mathrm d\lambda_t\,\mathrm dt
    +\int_{\mathbb R}v_0(x)\eta(0,x)\,\mathrm d\lambda_0=0
       \label{eq:line-momentum}
\end{align}
for compactly supported $C^1$ tests $\zeta$ and $\eta$ on
$[0,\infty)\times\R$. The entropy requirements are
\begin{equation}\label{eq:line-entropy}
 v(t,y)-v(t,x)\leq\frac{y-x}{t}\quad(x<y)
 \quad t\in(0,\infty)\setminus N,
 \qquad
 v_t^2\lambda_t\rightharpoonup v_0^2\lambda_0\quad(t\downarrow0).
\end{equation}
Here $N\subset(0,\infty)$ is a Lebesgue-null set independent
of $x$ and $y$.
Symmetry means that $\lambda_t$ is even and $p_t$ is odd. This definition
is in terms of the Euler equations and entropy conditions; it does not
prescribe a particle or projection representation.

Reflection must preserve the weak identities, representatives, and entropy
conditions to transfer one-dimensional uniqueness. Put $j_+(r)=r$ and $j_-(r)=-r$, and set
\begin{align}
 \lambda_t&=\frac12\bigl((j_+)_\#\mu_t+(j_-)_\#\mu_t\bigr),
       \label{eq:reflection-mass}\\
 v(t,x)&=\operatorname{sgn}(x)u(t,|x|)\quad(x\ne0),
 \qquad v(t,0)=0.\label{eq:reflection-velocity}
\end{align}
Writing $p_t=v(t,\cdot)\lambda_t$ and
$k_t=v(t,\cdot)^2\lambda_t$, we obtain
\begin{equation}\label{eq:reflection-momentum-energy}
 p_t=\tfrac12\bigl((j_+)_\#q_t-(j_-)_\#q_t\bigr),
 \qquad
 k_t=\tfrac12\bigl((j_+)_\#e_t+(j_-)_\#e_t\bigr).
\end{equation}
In particular, $\lambda_t(\{0\})=\mu_t(\{0\})$ and, when finite, the
total masses of $\lambda_t$ and $\mu_t$ agree. The two reflected origin
atoms are each multiplied by $1/2$.

Parity matches the tests and traces, while OSL across the origin follows
from the radial inequality with one endpoint at zero.

\begin{theorem}[Reflection equivalence]\label{thm:reflection}
Fix radial initial data \eqref{eq:radial-data}, and let $(\lambda_0,v_0)$
be their reflection under
\eqref{eq:reflection-mass}--\eqref{eq:reflection-velocity}.
The same transformation defines a bijection, at the level of mass and
momentum, between
\begin{itemize}
\item radial entropy solutions with these data in the sense of
Definition~\ref{def:entropy};
\item symmetric one-dimensional entropy solutions of
\eqref{eq:line-mass}--\eqref{eq:line-entropy} with data $(\lambda_0,v_0)$.
\end{itemize}
Its inverse on mass is $\mu_t=|\cdot|_\#\lambda_t$.
The correspondence preserves the velocity bound and the initial
kinetic-energy trace. Through Proposition~\ref{prop:radial}, it also
gives a bijection, at the level of mass and momentum, with full-space
radial entropy solutions in the sense of
Definition~\ref{def:full-entropy}, for the corresponding initial data.
\end{theorem}

\begin{proof}
For a radial entropy solution $(\mu,u)$, define $(\lambda,v)$ by
\eqref{eq:reflection-mass}--\eqref{eq:reflection-velocity}.
If $\zeta$ and $\eta$ are $C^1$ tests on the line, their even and odd
restrictions are
\[
 \zeta_e(t,r)=\frac{\zeta(t,r)+\zeta(t,-r)}2,
 \qquad
 \eta_o(t,r)=\frac{\eta(t,r)-\eta(t,-r)}2.
\]
The identities $(\zeta_e)_r(t,0)=0$ and $\eta_o(t,0)=0$ allow these tests
in the radial weak equations. By
\eqref{eq:reflection-mass}--\eqref{eq:reflection-momentum-energy},
\begin{equation}\label{eq:reflection-weak-pairings}
\begin{aligned}
 \int_{\mathbb R}(\zeta_t+v\zeta_x)\,\mathrm d\lambda_t
  &=\int_{[0,\infty)}((\zeta_e)_t+u(\zeta_e)_r)\,\mathrm d\mu_t,\\
 \int_{\mathbb R}(v\eta_t+v^2\eta_x)\,\mathrm d\lambda_t
  &=\int_{[0,\infty)}(u(\eta_o)_t+u^2(\eta_o)_r)\,\mathrm d\mu_t.
\end{aligned}
\end{equation}
The identities in \eqref{eq:reflection-weak-pairings} remain valid
at zero because $u(t,0)=v(t,0)=0$; in particular,
the mass atom retains its full weight. Applying the same parity
calculation to the initial terms proves
\eqref{eq:line-mass}--\eqref{eq:line-momentum}.

For continuous compactly supported line tests $\zeta$ and $\eta$,
let $\zeta_e$ and $\eta_o$ denote their even and odd restrictions
to $[0,\infty)$, respectively.
The time-dependent pairings satisfy
\begin{equation}\label{eq:reflection-test-pairings}
 \begin{gathered}
 \int_{\mathbb R}\zeta\,\mathrm d\lambda_t
   =\int_{[0,\infty)}\zeta_e\,\mathrm d\mu_t,\qquad
 \int_{\mathbb R}\eta\,\mathrm dp_t
   =\int_{[0,\infty)}\eta_o\,\mathrm dq_t,\\
 \int_{\mathbb R}\zeta\,\mathrm dk_t
   =\int_{[0,\infty)}\zeta_e\,\mathrm de_t.
 \end{gathered}
\end{equation}
The mass and momentum identities in \eqref{eq:reflection-test-pairings}
give the weakly continuous line representatives, since $\eta_o(0)=0$.
These correspondences are reversible: even
extension supplies every continuous radial mass test, and odd extension
supplies every continuous radial momentum test vanishing at zero.
The kinetic-measure identity in \eqref{eq:reflection-test-pairings}
gives the equivalence of initial kinetic-energy traces by even extension. It applies without change to an initial
origin atom.

Fix a positive time at which \eqref{eq:radial-osl} holds for every pair
of radii. Its positive half-line restriction is the required line OSL
bound. On the negative half-line, for $0\leq a<b$,
\[
 v(t,-a)-v(t,-b)=u(t,b)-u(t,a)\leq(b-a)/t.
\]
Taking $r_1=0$ in \eqref{eq:radial-osl} also gives
\begin{equation}\label{eq:boundary-osl}
 u(t,r)\leq r/t\qquad(r\geq0).
\end{equation}
Consequently, for $a,b\geq0$,
\[
 v(t,b)-v(t,-a)=u(t,b)+u(t,a)\leq(a+b)/t.
\]
This covers pairs on opposite sides of zero and proves the line OSL bound for
almost every positive time. The reflection preserves boundedness and
joint Borel measurability, so $(\lambda,v)$ belongs to the stated line
entropy class. The cross-origin estimate uses the radial OSL inequality with one endpoint at $r=0$.

Conversely, let $(\lambda,v)$ be a line entropy solution with even mass
$\lambda_t$ and odd momentum $p_t$. These symmetries determine the
velocity only up to a mass-null set. For a compactly supported
continuous test $\eta$,
\[
 \int_{\mathbb R}\eta(x)v(t,-x)\,\mathrm d\lambda_t(x)
 =\int_{\mathbb R}\eta(-x)\,\mathrm dp_t(x)
 =-\int_{\mathbb R}\eta(x)\,\mathrm dp_t(x).
\]
Thus $v(t,-x)=-v(t,x)$ for $\lambda_t$-almost every $x$, at every time.
The representative
\[
 \widetilde v(t,x)=\frac{v(t,x)-v(t,-x)}2
\]
is odd, bounded, and jointly Borel, and agrees with $v$ almost everywhere
with respect to $\lambda_t$. In particular it keeps the value zero at
any positive-mass origin atom and changes neither $p_t$ nor $k_t$.
At every time for which the original all-pairs OSL bound holds,
\[
 \widetilde v(t,y)-\widetilde v(t,x)
 =\tfrac12\bigl(v(t,y)-v(t,x)+v(t,-x)-v(t,-y)\bigr)
 \leq(y-x)/t\qquad(x<y).
\]
Hence this representative has the same entropy properties, including
the initial kinetic-energy trace.

Set $\mu_t=|\cdot|_\#\lambda_t$ and
$u(t,r)=\widetilde v(t,r)$ for $r>0$, with $u(t,0)=0$.
Evenness of $\lambda_t$ recovers \eqref{eq:reflection-mass}; oddness of
$\widetilde v$ recovers \eqref{eq:reflection-velocity}.
The even extension of a compatible radial mass test and the odd
extension of a radial momentum test are compactly supported $C^1$
line tests. Inserting them into \eqref{eq:line-mass}--\eqref{eq:line-momentum}
and using \eqref{eq:reflection-weak-pairings} for the spacetime terms
and \eqref{eq:reflection-test-pairings} at $t=0$ for the initial terms
gives the radial equations. The pairings in
\eqref{eq:reflection-test-pairings} give
the required time representatives and initial kinetic-energy trace. Finally,
the line OSL bound restricted to $0\leq x<y$, with
$\widetilde v(t,0)=0$, is \eqref{eq:radial-osl}. Thus $(\mu,u)$ is a
radial entropy solution. Reflection and restriction are inverse on mass
and momentum, and Proposition~\ref{prop:radial} supplies the full-space
correspondence.
\end{proof}

Permanent absorption is proved separately in
Proposition~\ref{prop:sticking}, using the mass balance and OSL at positive
times and the initial kinetic-energy trace at $t=0$.

\subsection[Proof of Theorem~\ref{thm:main}]{Proof of Theorem~\ref{thm:main}}\label{sec:preliminaries}

We apply the theorem of Huang and Wang
\cite[Definitions~1--2 and Theorems~1--2; p.~145, after (4.36)]{HW2001}, stated here
in its mass-primitive formulation.

\begin{theorem}[Huang--Wang existence and uniqueness]
\label{thm:hw-tool}
Let $\lambda_0\in\mathcal M^+_{\mathrm{loc}}(\mathbb R)$ and let
$v_0$ be a bounded $\lambda_0$-measurable function. There is a pair
$(m,v)$ with $m(t,\cdot)\in BV_{\mathrm{loc}}(\mathbb R)$,
$\lambda_t=\partial_xm(t,\cdot)\ge0$, and bounded mass-measurable
velocity $v$, such that the curves $\lambda_t$ and $p_t=v_t\lambda_t$
are weakly continuous on $[0,\infty)$ against compactly supported
continuous tests and have initial values $\lambda_0$ and
$v_0\lambda_0$. For compactly supported smooth tests on
$(0,\infty)\times\mathbb R$,
\[
 \int_0^\infty\!\int_{\mathbb R}m\zeta_t\,\mathrm dx\,\mathrm dt
 -\int_0^\infty\!\int_{\mathbb R}v\zeta\,\mathrm d\lambda_t\,\mathrm dt=0,
 \qquad
 \int_0^\infty\!\int_{\mathbb R}
 (v\eta_t+v^2\eta_x)\,\mathrm d\lambda_t\,\mathrm dt=0.
\]
The velocity has a representative for which there is a
Lebesgue-null set $N\subset(0,\infty)$, independent of $x$ and $y$, such that
for every $t\in(0,\infty)\setminus N$,
\[
 v(t,y)-v(t,x)\le\frac{y-x}{t}
 \quad\text{for all }x<y,
\]
and as $t\in(0,\infty)$ tends to $0$,
\[
 v_t^2\lambda_t\rightharpoonup v_0^2\lambda_0.
\]
The constructed solution can be chosen with
$|v|\le\|v_0\|_{L^\infty(\lambda_0)}$.
Any two pairs with these weak equations, time representatives, and
entropy conditions have the same mass and momentum for every
$t\ge0$, and their velocities agree almost everywhere with respect to
$\lambda_t$ for each such $t$. No finite-mass or moment assumption is required.
\end{theorem}

The original uniqueness conclusion is a space--time almost-everywhere statement.
For two solutions with the same initial data and any $\varphi\in C_c(\mathbb R)$,
the functions
\[
 t\longmapsto\int_{\mathbb R}\varphi\,\mathrm d(\lambda_t^1-\lambda_t^2),
 \qquad
 t\longmapsto\int_{\mathbb R}\varphi\,\mathrm d(p_t^1-p_t^2)
\]
are continuous and vanish for almost every $t>0$, hence vanish
for every $t\ge0$.

Lemma~\ref{lem:normalized-primitive} supplies the normalized
mass primitive. Lemma~\ref{lem:osl-representative} provides a jointly
Borel velocity representative that agrees with the original velocity
$\lambda_t$-almost everywhere at each time and satisfies the all-pairs OSL
bound outside a common null set of positive times.

The next lemma fixes the time-dependent additive constant
in the normalized mass primitive required by Theorem~\ref{thm:hw-tool}, so that
both the primitive equation and the prescribed initial trace hold,
without imposing a total-mass assumption.

\begin{lemma}[A normalized mass primitive]\label{lem:normalized-primitive}
Let $\lambda:[0,\infty)\to\Mloc^+(\R)$ be locally weakly continuous,
and let $v$ be jointly Borel with $|v|\leq L<\infty$.
If \eqref{eq:line-mass} holds, there is a jointly Borel function $m$
with
\[
 \begin{gathered}
 m\in L^\infty_{\mathrm{loc}}([0,\infty)\times\R)
       \cap C([0,\infty);L^1_{\mathrm{loc}}(\R)),\\
 m(t,\cdot)\in BV_{\mathrm{loc}}(\R)\quad(t\geq0),\qquad
 m\in BV_{\mathrm{loc}}((0,\infty)\times\R),\\
 \partial_xm(t,\cdot)=\lambda_t\quad(t\geq0),\qquad
 \partial_tm=-v\lambda_t\,\dd t.
 \end{gathered}
\]
The time derivative identity is understood as an identity of
space--time measures. The initial $L^1_{\mathrm{loc}}$ trace of $m$ is the signed cumulative
mass primitive $F_0(x)=\lambda_0([0,x))$ for $x>0$ and
$F_0(x)=-\lambda_0([x,0))$ for $x<0$, with $F_0(0)=0$.
Conversely, these derivative identities and initial trace imply
\eqref{eq:line-mass}.
\end{lemma}

\begin{proof}
\emph{Step 1. A locally bounded cumulative mass primitive.}
Choose a nonnegative
$\chi\in C_c^\infty(\R)$ with $\int\chi\,\dd x=1$, and fix the signed
cumulative mass primitive
\[
 F(t,x)=
 \begin{cases}
  \lambda_t([0,x)),&x>0,\\
  0,&x=0,\\
  -\lambda_t([x,0)),&x<0.
 \end{cases}
\]
The half-open interval convention gives $\partial_xF=\lambda$,
including the atom at zero. Since the weakly continuous measure curve
is a Borel kernel and the intervals depend measurably on their endpoints,
$F$ is jointly Borel. Testing $\lambda_t$ against a compactly supported
continuous function equal to one on $[-R,R]$ gives, for every $R,T>0$,
\[
 K_{R,T}:=\sup_{0\le t\le T}\lambda_t([-R,R])<\infty,
 \qquad |F(t,x)|\le K_{R,T}\quad(|x|\le R).
\]
With $F_0=F(0,\cdot)$, define
\begin{equation}\label{eq:normalized-spatial-primitive}
 \begin{split}
 G(t,x)&=F(t,x)-\int_\R\chi(y)F(t,y)\,\dd y,\\
 m(t,x)&=G(t,x)+\int_\R\chi(y)F_0(y)\,\dd y
       -\int_0^t\int_\R\chi(y)v(s,y)\,\lambda_s(\dd y)\,\dd s.
 \end{split}
\end{equation}
If
$\operatorname{supp}\chi\subset[-R,R]$ and $L=\|v\|_\infty$, then
$|m(t,x)|\le(3+TL\|\chi\|_\infty)K_{R,T}$ for
$0\le t\le T$ and $|x|\le R$.

\emph{Step 2. Removal of the spatially constant defect.}
The spatial derivative of $G$ is $\lambda$, and the continuity equation
therefore yields, in distributions,
\[
 \partial_x(\partial_tG+v\lambda)
 =\partial_t\lambda+\partial_x(v\lambda)=0.
\]
Thus $\partial_tG+v\lambda$ is constant in space. Its pairing with
$\chi$ equals $\int\chi v\,\dd\lambda_t$, because
$\int\chi G=0$. The time integral in
\eqref{eq:normalized-spatial-primitive} removes exactly this constant:
\[
 \partial_xm=\lambda_t\,\dd t,
 \qquad \partial_tm=-v\lambda_t\,\dd t
\]
as space--time measures. By the construction in
\eqref{eq:normalized-spatial-primitive}, $\partial_xm(t,\cdot)=\lambda_t$ and
$m(t,\cdot)\in BV_{\mathrm{loc}}(\R)$ for every $t\ge0$.
The space--time derivative identities and the local bound on $m$
give $m\in BV_{\mathrm{loc}}((0,\infty)\times\R)$. On
$Q=(0,T)\times(-R,R)$ they give the variation estimate
\[
 |D_xm|(Q)+|D_tm|(Q)
 \le (1+L)\int_0^T\lambda_t([-R,R])\,\dd t.
\]

\emph{Step 3. Time continuity and the initial trace.}
The same normalization determines the time representative. Indeed,
\[
 G(t,x)=\int_\R
 \left(\mathbf1_{\{y<x\}}-\int_y^\infty\chi(z)\,\dd z\right)
 \lambda_t(\dd y).
\]
For fixed $x$, the kernel is compactly supported and its only possible
discontinuity is at $y=x$. Weak continuity of $\lambda$ implies
$G(t,x)\to G(s,x)$ as $t\to s$ whenever $\lambda_s(\{x\})=0$.
The exceptional set of $x$ is at most countable. Dominated convergence,
using $|F(t,x)|\leq K_{R,T}$ and
$G(t,x)=F(t,x)-\int\chi F(t,\cdot)$ on an interval $[-R,R]$
containing $\operatorname{supp}\chi$, gives continuity of $G$ in
$L^1_{\mathrm{loc}}(\R)$. The time integral in
\eqref{eq:normalized-spatial-primitive} is continuous and vanishes at
zero. Since $m(0,\cdot)=F_0$, we obtain
$m(t,\cdot)\to F_0$ in $L^1_{\mathrm{loc}}(\R)$ as $t\downarrow0$.
The relation $\partial_tm=-v\lambda_t\,\dd t$ and the initial
$L^1_{\mathrm{loc}}$ trace give, for tests meeting $t=0$,
\[
 \int_0^\infty\!\int_\R m\zeta_t\,\dd x\,\dd t
 -\int_0^\infty\!\int_\R v\zeta\,\dd\lambda_t\,\dd t
 +\int_\R F_0(x)\zeta(0,x)\,\dd x=0.
\]
This is the primitive equation in Theorem~\ref{thm:hw-tool}, with a
fixed initial primitive. Conversely, differentiation in $x$ recovers
\eqref{eq:line-mass}, including its initial term. The normalization
changes no mass, momentum, OSL bound or initial kinetic-energy trace.

\end{proof}

Suder and Tudorascu constructed a jointly Borel
representative satisfying the everywhere Oleinik condition for
sticky-particle flow (SPF) and scalar conservation law (SCL) solutions
with initial mass in $\mathcal P_2(\R)$ and bounded,
right-continuous initial velocity
\cite{ST2023}.
We require a representative-selection result for locally finite mass and
momentum curves, without assuming a Lagrangian representation or right
continuity of the initial velocity. We therefore give a direct
measure-theoretic proof. The construction preserves the velocity almost
everywhere with respect to the mass measure at every time and enforces
the spatial OSL bound on a common Borel set of full measure in
positive time.

\begin{lemma}[A Borel all-pairs OSL representative]\label{lem:osl-representative}
Let $\lambda_t\geq0$ and $p_t$ be locally weakly continuous Radon
measure curves on $[0,\infty)$, with $|p_t|\leq L\lambda_t$ for
all $t\geq0$. Suppose that, for almost every $t>0$, their
Radon--Nikodym velocity satisfies
\[
 v_t(y)-v_t(x)\leq(y-x)/t
 \quad\text{for }\lambda_t\otimes\lambda_t\text{-almost every }x<y.
\]
There is a jointly Borel representative $\widetilde v$, bounded by
$L$, such that $p_t=\widetilde v_t\lambda_t$ for every $t\geq0$
and the inequality holds for all $x<y$ at almost every positive time.
It preserves the velocity at every positive-mass atom and the kinetic
measure at each time, and hence preserves any prescribed initial traces.
\end{lemma}

\begin{proof}
\emph{Step 1. A jointly Borel Radon--Nikodym velocity.}
For $j\geq1$, form the Borel functions
\[
 v_j(t,x)=
 \begin{cases}
 \displaystyle\frac{p_t((x-j^{-1},x+j^{-1}))}
                      {\lambda_t((x-j^{-1},x+j^{-1}))},
        &\lambda_t((x-j^{-1},x+j^{-1}))>0,\\[4pt]
 0,&\text{otherwise}.
 \end{cases}
\]
Borel measurability follows from integration of the interval indicators
against the weakly continuous measure kernels; $|v_j|\leq L$.
Define $v(t,x)$ as the limit of $v_j(t,x)$ when it exists, and as zero
otherwise. Differentiation of Radon measures on $\R$ gives
$v_t=\dd p_t/\dd\lambda_t$ $\lambda_t$-almost everywhere for every
$t$, so this bounded Borel representative preserves $p_t$ at every time.
At $t=0$ it may be replaced by any prescribed Borel representative
of the same initial velocity bounded by $L$.

\emph{Step 2. Spatial OSL extension and preservation of the velocity.}
For $t>0$, put
\[
 w_t(y)=v_t(y)-y/t,\qquad
 G_t(x)=\underset{y\geq x}{\operatorname{ess\,sup}}w_t(y),\qquad
 \widehat v_t(x)=\max\{-L,\min\{L,x/t+G_t(x)\}\},
\]
where the essential supremum is with respect to $\lambda_t$ and is
$-\infty$ on a null tail. The function $G_t$ is nonincreasing, with
$G_t(x)\leq L-x/t$. Since
\[
 \widehat v_t(x)-x/t
 =\max\{-L-x/t,\min\{L-x/t,G_t(x)\}\}
\]
is nonincreasing, $\widehat v$ satisfies the all-pairs bound and
$|\widehat v|\leq L$.

At a time when the assumed almost-everywhere bound holds, Fubini gives
$G_t(x)\leq w_t(x)$ for $\lambda_t$-almost every $x$; the diagonal
adds no restriction. For the reverse inequality, let
$E_c=\{x:w_t(x)>c\geq G_t(x)\}$, with $c$ rational.
Every $x\in E_c$ has $\lambda_t(E_c\cap[x,\infty))=0$.
If $\lambda_t(E_c)>0$, choose $N$ so that
$E=E_c\cap[-N,N]$ has positive finite mass. Integration over $E$
contradicts
\[
 (\lambda_t\otimes\lambda_t)
       \{(x,y)\in E^2:y\geq x\}\geq\tfrac12\lambda_t(E)^2>0.
\]
The countable union over $c$ proves $w_t=G_t$ almost everywhere,
and hence $\widehat v_t=v_t$ almost everywhere. In particular, every
atom keeps its velocity; the use of the closed tail includes its mass.

\emph{Step 3. Measurability and the common set of positive times.}
For each rational $c$,
\[
 \{(t,x):G_t(x)>c\}
 =\left\{(t,x):\int_\R
   \mathbf1_{\{y\geq x,\ w_t(y)>c\}}\,\lambda_t(\dd y)>0\right\}
\]
is Borel. The set of times at which the assumed bound holds is also
Borel: for each integer $N$, integrate the indicator of its violation
over $[-N,N]^2$ against $\lambda_t\otimes\lambda_t$ and require the
result to vanish. On this common full-measure time set choose
$\widetilde v_t=\widehat v_t$; elsewhere retain $v_t$, including at
zero. The representative is jointly Borel and agrees with the
Radon--Nikodym velocity at every time. Momentum, kinetic measures and
all their prescribed traces are therefore unchanged.
\end{proof}

We now apply the Eulerian theorem to the reflected data
and use uniqueness to recover the required symmetry.

\begin{proof}[Existence and uniqueness in Theorem~\ref{thm:main}]
\emph{Step 1. Existence from the reflected initial data.}
The full-space initial data give $\mu_0=|\cdot|_\#\rho_0$ and the scalar
radial velocity $u_0$, satisfying \eqref{eq:radial-data}.
Apply Theorem~\ref{thm:hw-tool} to the reflected initial data in
\eqref{eq:reflection-mass}--\eqref{eq:reflection-velocity}.
Differentiating its primitive identity in $x$ gives the mass equation
on positive times. The weakly continuous mass and momentum traces
supply the initial terms in \eqref{eq:line-mass}--\eqref{eq:line-momentum}.
Lemma~\ref{lem:osl-representative} gives a bounded jointly Borel
all-pairs OSL velocity $w$ without changing any mass, momentum or
kinetic measure. Thus $(\lambda,w)$ belongs to the line entropy
class in Section~\ref{subsec:reflection-equivalence}.

\emph{Step 2. Symmetry and radial reconstruction.}
Reflection acts on a whole-line solution by
\[
 (\lambda_t,v(t,x))\longmapsto
 ((x\mapsto-x)_\#\lambda_t,-v(t,-x)),
\]
and preserves the reflected initial data.
Lemma~\ref{lem:normalized-primitive} places the reflected Euler pair
in the same primitive class. The uniqueness in Theorem~\ref{thm:hw-tool}
therefore makes $\lambda_t$ even and $p_t=w_t\lambda_t$ odd at every
time. For the all-pairs OSL representative $w$, set
\[
 v(t,x)=\tfrac12\bigl(w(t,x)-w(t,-x)\bigr).
\]
The symmetry of $\lambda_t$ and $p_t$ gives $v=w$
$\lambda_t$-almost everywhere at each time. The odd representative
$v$ is jointly Borel, has the same bound, satisfies $v(t,0)=0$,
and retains the OSL bound, as verified in
Theorem~\ref{thm:reflection}. In particular, its mass, momentum and
kinetic measures have the original initial traces.
Theorem~\ref{thm:reflection} gives the inverse reflection, and
Proposition~\ref{prop:radial} reconstructs the required full-space
radial entropy solution $(\rho,\mathbf u)$ of \eqref{eq:euler}.

\emph{Step 3. Uniqueness at every time.}
For two full-space radial entropy solutions in the sense of
Definition~\ref{def:full-entropy} with the same initial data,
Proposition~\ref{prop:radial} and Theorem~\ref{thm:reflection}
give symmetric whole-line entropy solutions with the same reflected
initial data. Lemma~\ref{lem:normalized-primitive} supplies these
solutions with normalized mass primitives having the same initial trace.
Theorem~\ref{thm:hw-tool} then identifies the reflected mass and momentum. Equality holds
at every time by their weakly continuous representatives.
The inverse reflection determines $\mu_t$ and $q_t=u_t\mu_t$ on
$(0,\infty)$ and the absorbed mass; both radial scalar momentum measures are zero on
$\{0\}$. Thus the two radial mass--momentum pairs agree on the closed
half-line, and their Radon--Nikodym velocities agree
$\mu_t$-almost everywhere. By Proposition~\ref{prop:radial}, their
full-space masses $\rho_t$ and vector momentum measures $\mathbf u(t,\cdot)\rho_t$
agree at every time, and their velocities agree $\rho_t$-almost
everywhere. Permanent absorption is proved in Proposition~\ref{prop:sticking}.
\end{proof}

\section[Absorption at the origin: structure and stability]{Absorption at the origin: structure and stability}\label{sec:absorption}

We first prove permanent absorption for locally finite mass.
Under the additional finite-mass and second-moment assumptions,
Subsections~\ref{subsec:radial-projection}--\ref{subsec:boundary-momentum}
establish Theorem~\ref{thm:main-absorption},
Subsection~\ref{subsec:absorption-stability} proves
Proposition~\ref{thm:main-stability}, and
Subsection~\ref{subsec:absorption-examples} illustrates
the different types of absorption.

\subsection[Permanent absorption at the origin]{Permanent absorption at the origin}

To prove absorption, symmetry must be supplemented by the entropy
conditions: OSL controls positive times, while the initial kinetic-energy trace controls
the initial atom.

\begin{proposition}[Absorption at the origin]\label{prop:sticking}
Let $(\mu,u)$ be the radial representation of the solution in
Theorem~\ref{thm:main}. Then
\[
 \mu_t(\{0\})\ge\mu_s(\{0\})\qquad(0\le s\le t).
\]
Mass reaching the origin at a positive time remains there, as does
the initial atom $\mu_0(\{0\})\delta_0$.
\end{proposition}

\begin{proof}
\emph{Step 1. Monotonicity at positive times.}
Choose $h\in C_c^1([0,\infty))$ with $0\le h\le1$, $h(0)=1$,
$h'\le0$, and $\operatorname{supp}h\subset[0,1]$, taking $h$ constant
near zero. For $c>0$ set
\[
 J_c(t)=\int_{[0,\infty)}h\!\left(\frac r{ct}\right)\mu_t(\dd r),
 \qquad t>0.
\]
On a compact time interval in $(0,\infty)$, the test
$h(r/(ct))$ is compatible at zero and has uniformly compact spatial
support. The weak mass equation therefore makes $J_c$ absolutely
continuous, with
\begin{equation}\label{eq:cone-mass-monotonicity}
 J_c'(t)=\int_{[0,\infty)}
 \frac{h'(r/(ct))}{ct}\left(u(t,r)-\frac rt\right)\mu_t(\dd r)
 \ge0
 \quad\text{for almost every }t>0.
\end{equation}
The sign follows from $h'\le0$ and $u(t,r)\le r/t$.
For $0<s<t$ this gives $J_c(t)\ge J_c(s)$. Dominated convergence
as $c\downarrow0$ yields
\[
 \mu_t(\{0\})\ge\mu_s(\{0\}).
\]

\emph{Step 2. Persistence of the initial origin atom.}
To compare with the initial atom, keep $c>0$ fixed and use the
time-independent test $h(r/(cs))$ in the mass balance on $[0,s]$.
The initial integral is at least $M_0:=\mu_0(\{0\})$, while
Cauchy--Schwarz in space and time gives
\begin{align}
 J_c(s)&\ge M_0-\frac{\|h'\|_\infty}{cs}
       \int_0^s\!\int_{[0,cs]}|u(\tau,r)|\,\mu_\tau(\dd r)\,\dd\tau
       \notag\\
 &\ge M_0-\frac{\|h'\|_\infty}{c}
       \left(C\sup_{0<\tau<s}e_\tau([0,cs])\right)^{1/2}.
       \label{eq:initial-cone-mass}
\end{align}
Here $C$ is independent of sufficiently small $s$: weak continuity
bounds $\mu_\tau([0,1])$ uniformly, and $cs\le1$.

The shrinking energy term is controlled by the initial kinetic-energy trace.
For $\delta>0$, choose a continuous cutoff
$0\le\chi_\delta\le1$, equal to one on $[0,\delta]$ and supported in
$[0,2\delta]$. For $cs<\delta$ we have
$e_\tau([0,cs])\le\int\chi_\delta\,\dd e_\tau$. Thus the initial trace gives
\[
 \limsup_{s\downarrow0}\sup_{0<\tau<s}e_\tau([0,cs])
 \le\int\chi_\delta\,\dd e_0.
\]
The right-hand side tends to zero as $\delta\downarrow0$, since
$e_0(\{0\})=u_0(0)^2\mu_0(\{0\})=0$. Substitution into
\eqref{eq:initial-cone-mass}, followed by
\eqref{eq:cone-mass-monotonicity}, gives
\[
 J_c(t)\ge\liminf_{s\downarrow0}J_c(s)\ge M_0
 \qquad(t>0).
\]
Only now do we send $c\downarrow0$, obtaining
$\mu_t(\{0\})\ge\mu_0(\{0\})$. Together with the positive-time
comparison, this proves monotonicity on $[0,\infty)$.

At the level of measures, for every $s\ge0$ the nonnegative curve
$\mu_t-\mu_s(\{0\})\delta_0$, $t\ge s$, satisfies the mass balance
with initial measure $\mu_s-\mu_s(\{0\})\delta_0$. Indeed, the
subtracted atom is stationary and has zero velocity. This gives
persistence for mass measures, including the initial atom.
\end{proof}

\begin{samepage}
\emph{A corresponding trajectory estimate.}
If a nonnegative absolutely continuous curve $R$ satisfies
$R'(\tau)=u(\tau,R(\tau))$ for almost every $\tau\in[s,T]$, where
$s>0$, the OSL bound gives
$(R(\tau)/\tau)'=(\tau R'(\tau)-R(\tau))/\tau^2\le0$ almost everywhere.
Consequently,
\begin{equation}\label{eq:radial-path-bound}
 0\le R(t)\le \frac ts R(s)\qquad\text{$(s\le t\le T)$}.
\end{equation}
This proves that $R(s)=0$ implies $R(t)=0$ for every $t\in[s,T]$;
the argument has no initial-time implication because it divides by
$s$.
\end{samepage}

\begin{remark}\label{rem:initial-energy}
The initial kinetic-energy trace cannot be omitted even in the symmetric
class. For $\alpha,c>0$, the radial data $\mu_0=\alpha\delta_0$, $u_0=0$
admit the expanding weak solution $\mu_t=\alpha\delta_{ct}$,
$u(t,ct)=c$. A representative
$u(t,r)=\min\{r/t,c\}$ satisfies the radial OSL bound.
Mass and the compatible momentum have the required initial traces,
but $u_t^2\mu_t\rightharpoonup\alpha c^2\delta_0$.
It is therefore excluded by the initial kinetic-energy trace. In the reflected system
this is the symmetric splitting into two oppositely moving atoms.
\end{remark}

\subsection[Lagrangian projection representation]{Lagrangian projection representation}\label{subsec:radial-projection}

From this subsection onward, the initial radial measure has finite mass
$m=\mu_0([0,\infty))>0$ and finite second moment. We retain the bounded
initial scalar radial velocity assumed in Theorem~\ref{thm:main}.
These additional assumptions permit an $L^2$ projection representation
of the radial solution. The zero level set of the position map determines the absorbed mass;
the temporal behavior of that set will also be used to express the
boundary impulse and prove stability of absorption.

The position projection identifies the labels at the origin,
while the velocity projection records their stopping and the motion of
positive-radius mass clusters. We derive the radial representation below from
the one-dimensional sticky-particle theory of Natile and Savar\'e
by odd reflection in mass coordinates.
Theorem~\ref{thm:hw-tool} will identify the resulting flow with the
Eulerian entropy solution.

We use the initial mass as a label and set
\[
 R_0(a)=\inf\{r\geq0:\mu_0([0,r])>a\},\qquad
 U_0(a)=u_0(R_0(a)),\qquad 0<a<m.
\]
Thus $(R_0)_{\#}\dd a=\mu_0$, $R_0\in L^2(0,m)$, and $U_0$ is constant
almost everywhere on every level interval of $R_0$. In particular,
$U_0=0$ on $\{R_0=0\}$. Set
\[
 \Kplus=\{R\in L^2(0,m):R\geq0\text{ and }R
                  \text{ has a nondecreasing representative}\}.
\]
This cone is convex and closed: an $L^2$-convergent sequence has a
subsequence converging almost everywhere, and its nonnegativity and
monotonicity pass to a representative of the limit.
The velocity must be constant on each positive-radius mass cluster and vanish on
the absorbed labels. To express these constraints by orthogonal
velocity projection, for $R\in\Kplus$, define the closed subspace
\[
 H_R^0:=\left\{W\in L^2(0,m):
 \begin{array}{l}
 W\text{ is }\sigma(R)\text{-measurable and}\\
 W=0\text{ a.e. on }\{R=0\}
 \end{array}\right\}.
\]
Equivalently, elements of $H_R^0$ are constant almost everywhere on
each positive-mass level interval of $R$, and vanish on the zero
level set.

\begin{theorem}[Radial projection representation]\label{thm:projection}
Let $(\mu,u)$ be the radial representation of the solution in
Theorem~\ref{thm:main}, under the finite-mass and second-moment
assumptions of this subsection. For every
$t\geq0$, its Lagrangian position and velocity have the following
representations.
\begin{enumerate}
\item[\textup{(i)}] The position map is
\begin{equation}\label{eq:radial-projection}
 R_t=P_{\Kplus}(R_0+tU_0),\qquad
 \mu_t=(R_t)_{\#}\dd a.
\end{equation}
\item[\textup{(ii)}] The right derivative of $R_t$ exists in $L^2(0,m)$
and equals
\begin{equation}\label{eq:radial-actual-velocity}
 V_t=\frac{\dd^+}{\dd t}R_t=P_{H_{R_t}^0}U_0,
 \qquad V_t=u(t,R_t)\quad\text{almost everywhere}.
\end{equation}
In particular, $V_t=0$ almost everywhere on $\{R_t=0\}$.
\end{enumerate}
If $r>0$ and $I=(a,b)$ is the entire level interval $\{R_t=r\}$,
up to null sets, with $b>a$, then
\begin{equation}\label{eq:full-shell-barycentres}
 r=\frac1{b-a}\int_a^b(R_0+tU_0)\,\dd c,
 \qquad
 u(t,r)=\frac1{b-a}\int_a^b U_0\,\dd c.
\end{equation}
\end{theorem}

\begin{proof}
\emph{Step 1. The one-dimensional projection input.}
On the signed label interval $(-m,m)$, equip mass coordinates with
the reference measure $\beta=\tfrac12\dd a$ and define
\[
 X_0(a)=\operatorname{sgn}(a)R_0(|a|),\qquad
 W_0(a)=\operatorname{sgn}(a)U_0(|a|).
\]
Then $\lambda_0=(X_0)_\#\beta$ has mass $m$, and $W_0$ is
$X_0$-measurable, including on its central zero level interval.
Here $K$ denotes the cone of nondecreasing maps in
$L^2((-m,m),\beta)$, and $H_X$ is the closed subspace of
$\sigma(X)$-measurable functions in this space.
The probability normalization entails no restriction for a finite
measure of mass $m>0$. On a finite interval $J$, write
$\beta=c\,\dd a$ with $c|J|=m$ and use an increasing affine
coordinate $z:J\to(0,1)$. Then $\dd\beta=m\,\dd z$ and
\[
 \int_J|X-Y|^2\,\dd\beta
   =m\int_0^1|X\circ z^{-1}-Y\circ z^{-1}|^2\,\dd z.
\]
The factor $m$ changes neither the metric position projection
nor the orthogonal velocity projection.
Applying the results of Natile and Savar\'e \cite[Theorems~2.3 and 2.6]{NS2009} to $\lambda_0/m$
therefore yields the same projection identities on
$L^2(J,\beta)$ and recovers total mass $m$ after push-forward.
On the signed interval above, the reflected sticky flow has
\[
 \begin{gathered}
 X_t=P_K(X_0+tW_0),\qquad \lambda_t=(X_t)_\#\beta,\\
 W_t=\frac{\dd^+}{\dd t}X_t=P_{H_{X_t}}W_0
       =\mathbb E[W_0\mid\sigma(X_t)].
 \end{gathered}
\]
The right derivative exists in $L^2(\beta)$ at every $t\geq0$.
The map $t\mapsto X_t$ is Lipschitz in $L^2(\beta)$, and
$t\mapsto W_t$ is strongly right-continuous there. In particular,
$(X_t,W_t)\to(X_0,W_0)$ in $L^2(\beta)^2$ as $t\downarrow0$.
With $W_t=v_t\circ X_t$, the pair $(\lambda_t,v_t\lambda_t)$ satisfies
\eqref{eq:line-mass}--\eqref{eq:line-momentum}, including their initial
terms. The mass is Lipschitz continuous in the $2$-Wasserstein distance,
and the momentum is weakly continuous as a finite signed measure.
For every $t>0$,
\[
 v_t(y)-v_t(x)\leq\frac{y-x}{t}
 \quad\text{for }\lambda_t\otimes\lambda_t\text{-almost every }(x,y)
 \text{ with }x<y.
\]
If $|W_0|\leq L$, conditional expectation gives
$|W_t|\leq L$ and hence $|v_t|\leq L$ $\lambda_t$-almost everywhere
at every time.

\emph{Step 2. Identification with the Eulerian entropy solution.}
The initial kinetic-energy trace follows from the strong $L^2$
convergence:
\[
 \|W_t^2-W_0^2\|_{L^1(\beta)}
 \leq(\|W_t\|_{L^2(\beta)}+\|W_0\|_{L^2(\beta)})
                         \|W_t-W_0\|_{L^2(\beta)}\longrightarrow0.
\]
For $\zeta\in C_c(\R)$, the functions $\zeta(X_t)$ are uniformly
bounded and converge in $\beta$-measure to $\zeta(X_0)$. Since
$W_0^2\in L^1(\beta)$, bounded convergence in measure with this fixed
integrable weight gives
\[
 \begin{aligned}
 &\left|\int\zeta\,v_t^2\,\dd\lambda_t
              -\int\zeta\,v_0^2\,\dd\lambda_0\right|\\
 &\quad\leq\|\zeta\|_\infty\|W_t^2-W_0^2\|_{L^1(\beta)}
       +\int|\zeta(X_t)-\zeta(X_0)|W_0^2\,\dd\beta
       \longrightarrow0.
 \end{aligned}
\]
Lemma~\ref{lem:osl-representative} supplies a jointly
Borel velocity representative, preserving $v_t$ $\lambda_t$-almost
everywhere at every time and satisfying the all-pairs OSL bound at
almost every positive time.
With Lemma~\ref{lem:normalized-primitive},
Theorem~\ref{thm:hw-tool} therefore identifies this flow with the reflection
of the radial representation of the solution in Theorem~\ref{thm:main}.

\emph{Step 3. Restriction to the radial position and velocity.}
The isometry $JX(a)=-X(-a)$ preserves $K$ and fixes $X_0+tW_0$.
Uniqueness of the metric position projection gives $JX_t=X_t$.
Restriction to $(0,m)$ maps the odd elements of $K$ onto
$\Kplus$: an odd nondecreasing map is nonnegative on the positive
labels, and the odd extension of a map in $\Kplus$ is
nondecreasing on $(-m,m)$. This restriction preserves the squared
norm, from $L^2((-m,m),\beta)$ to $L^2(0,m)$. The minimizing
property consequently gives
$R_t=P_{\Kplus}(R_0+tU_0)$. Reflection equivalence then yields
$\mu_t=(R_t)_\#\dd a$, proving \eqref{eq:radial-projection}.

The subspace $H_{X_t}$ is invariant under $J$, so $W_t$ is odd.
On the central level interval $\{X_t=0\}$, it is constant and hence
zero. The odd elements of $H_{X_t}$ restrict precisely to $H_{R_t}^0$;
restriction of the orthogonal velocity projection and of the right derivative
therefore gives \eqref{eq:radial-actual-velocity}. Conditional
expectation on the entire level interval at $r>0$ gives the velocity
identity in \eqref{eq:full-shell-barycentres}.

For the position identity, put $F=R_0+tU_0$ and $R=R_t$.
If $h$ is Lipschitz on $[0,\infty)$, has bounded support and
$h(0)=0$, then $r\mapsto r\pm\varepsilon h(r)$ is nondecreasing
and fixes zero whenever $\varepsilon\operatorname{Lip}(h)<1$.
Thus $R\pm\varepsilon h(R)\in\Kplus$. Applying the projection
variational inequality to these two competitors gives
\[
 \int_0^m(F-R)h(R)\,\dd a=0.
\]
For $r>0$, choose uniformly bounded Lipschitz functions supported
away from zero and converging pointwise to $\mathbf1_{\{r\}}$.
Passing to the limit in $\int_0^m(F-R)h(R)\,\dd a=0$ by dominated
convergence gives
$\int_{\{R=r\}}(F-r)\,\dd a=0$, proving the position formula on the
entire level interval.
\end{proof}

The projection describes positions. Formula
\eqref{eq:radial-actual-velocity} describes the velocity after
collisions and absorption; division of the displacement by $t$ does
not give this velocity. In particular, the average initial radial
velocity of absorbed labels need not vanish, although their current
velocity does.

\subsection[Absorbed mass and its temporal structure]{Absorbed mass and its temporal structure}

To extract the absorbed mass from the position projection, we
identify the endpoint of its zero level set.
The cone $\Kplus$ imposes both monotonicity and
nonnegativity, so the zero level set is an initial interval up to null
sets, whose length equals the absorbed mass.
The following lemma identifies the right endpoint of this
entire interval with the largest minimizer of the primitive of the
unprojected map. When the primitive has several minimizers, choosing a
smaller one would omit part of the absorbed mass.
We prove the endpoint identity below directly from the
variational characterization of projection onto the cone $\Kplus$ of nonnegative nondecreasing maps.

\begin{lemma}[Zero set of a cone projection]\label{lem:projection-zero}
For $F\in L^2(0,m)$, let $R=P_{\Kplus}F$ and define the primitive
$A(a)=\int_0^aF(b)\,\dd b$ on $[0,m]$. Then
\begin{equation}\label{eq:zero-set-projection}
 |\{R=0\}|=\max\Argmin_{a\in[0,m]}A(a).
\end{equation}
\end{lemma}

\begin{proof}
Set $Z=F-R$ and $b=|\{R=0\}|$, so that $R=0$ on $(0,b)$
and $R>0$ on $(b,m)$, up to null sets. The cone projection
inequality gives
\[
 \int_0^m ZR\,\dd a=0,\qquad
 \int_a^m Z\,\dd c\leq0\quad(0\leq a\leq m),
\]
by testing with $0$, $2R$ and $\mathbf1_{(a,m)}$.
The admissible competitor $(R-\varepsilon)_+$ also gives
\[
 \int_0^m Z\min\{R,\varepsilon\}\,\dd a\geq0.
\]
Since $\varepsilon^{-1}\min\{R,\varepsilon\}\to\mathbf1_{(b,m)}$
almost everywhere and is bounded by one, dominated convergence
and $\int_b^mZ\,\dd a\leq0$ yield $\int_b^mZ\,\dd a=0$.
Using $\int_a^m Z\,\dd c\leq0$ and $\int_b^mZ\,\dd a=0$, we compare
$A(b)$ with every value of $A$:
\[
 \begin{aligned}
 A(b)-A(a)&=\int_a^bZ\,\dd c\leq0
       &&(0\leq a\leq b),\\
 A(a)-A(b)&=\int_b^a(Z+R)\,\dd c
          \geq\int_b^aR\,\dd c>0
       &&(b<a\leq m).
 \end{aligned}
\]
Thus $b$ is a minimizer and no larger label is a minimizer. The
same argument covers $b=0$ and $b=m$.
\end{proof}

For the free-position map, the primitive in the lemma is affine
in time. We write it in terms of the continuous initial-position and initial-velocity
primitives:
\begin{equation}\label{eq:absorption-primitives}
 B(a)=\int_0^aR_0(b)\,\dd b,\qquad
 C(a)=\int_0^aU_0(b)\,\dd b,\qquad A_t(a)=B(a)+tC(a),
\end{equation}
and write $m_0=\mu_0(\{0\})$. Both $B$ and $C$ vanish on $[0,m_0]$;
if $m_0<m$, the function $B$ is strictly increasing on $[m_0,m]$
because $R_0>0$ almost everywhere on $(m_0,m)$.

Moutsinga \cite[Propositions~4.2--4.3]{Moutsinga2008}
proves temporal nesting and c\`adl\`ag regularity of cluster-label
endpoints, together with a joint space--time limit property, under a
negative-jump condition on the initial velocity. Here the strict
increase of $B$ on $[m_0,m]$ orders all minimizers at distinct times.
Minimization on $[m_0,m]$ identifies the right and left limits
of the absorbed mass, without that additional condition on $u_0$.

We say that absorption is continuous at $t>0$ if the absorbed mass
$M(t)=\mu_t(\{0\})$ is continuous there; its Stieltjes measure
may still have a singular continuous part.
We call absorption absolutely continuous on a time interval in
$(0,\infty)$ when $\dd M\ll\dd t$ there.

\begin{theorem}[Absorbed mass and its temporal structure]\label{thm:absorption}
Let $(\mu,u)$ be the radial representation in Theorem~\ref{thm:projection},
with $A_t,B,C$ defined in \eqref{eq:absorption-primitives} and
$m_0=\mu_0(\{0\})$. Its absorbed mass
$M(t)=\mu_t(\{0\})$ has the following properties.
\begin{enumerate}
\item[\textup{(i)}] For every $t\geq0$,
\begin{equation}\label{eq:absorption-argmin}
 M(t)=\max\Argmin_{a\in[0,m]}A_t(a)
     =\max\Argmin_{a\in[m_0,m]}A_t(a).
\end{equation}
The function $M$ is nondecreasing and right continuous on
$[0,\infty)$, with $M(0)=m_0$. For every $t>0$,
\begin{equation}\label{eq:absorption-left-limit}
 M(t-)=\min\Argmin_{a\in[m_0,m]}A_t(a).
\end{equation}
\item[\textup{(ii)}] At a positive time $t$, absorption is continuous
if and only if $A_t$ has a unique minimizer on $[m_0,m]$.
Writing $\alpha,\beta$ for its smallest and largest minimizers,
$M$ has a jump exactly when $\alpha<\beta$. Its size is then
$\beta-\alpha$, and
\begin{equation}\label{eq:absorption-jump-balance}
 \int_\alpha^\beta(R_0+tU_0)\,\dd a=0,\qquad
 \int_\alpha^a(R_0+tU_0)\,\dd b\geq0
       \quad(\alpha\leq a\leq\beta).
\end{equation}
\item[\textup{(iii)}] With $L=\|u_0\|_{L^\infty(\mu_0)}$, for every
$t\geq0$ one has
\begin{equation}\label{eq:absorption-finite-speed}
 m_0\leq M(t)\leq\mu_0([0,Lt]).
\end{equation}
\end{enumerate}
\end{theorem}

\begin{proof}
\emph{Step 1. The absorbed labels and the largest minimizer.}
The push-forward formula $\mu_t=(R_t)_\#\dd a$ and
Lemma~\ref{lem:projection-zero} give the first equality in
\eqref{eq:absorption-argmin}. Since $B=C=0$ on $[0,m_0]$,
a largest minimizer lies in $[m_0,m]$, which gives the second
equality. If $m_0=m$, then $A_t=0$ and $M(t)=m$ for all $t$,
so all the assertions hold. Suppose henceforth that $m_0<m$.

\emph{Step 2. Ordering of minimizers and temporal traces.}
For $0\leq s<t$, let $a$ minimize $A_s$ and let $b$ minimize
$A_t$, both on $[m_0,m]$. The inequalities
$A_s(a)\leq A_s(b)$ and $A_t(b)\leq A_t(a)$ are equivalent to
\[
 s\bigl(C(a)-C(b)\bigr)
 \leq B(b)-B(a)
 \leq t\bigl(C(a)-C(b)\bigr).
\]
Since $s<t$, these bounds give $C(a)\geq C(b)$ and
$B(b)\geq B(a)$. Strict increase of $B$ on $[m_0,m]$ then gives
$a\leq b$. This orders every earlier minimizer before every later
one; in particular, $M$ is nondecreasing.

The uniform convergence $A_s\to A_t$ as $s\to t$ implies that
any limit of minimizers of $A_s$ minimizes $A_t$. For $s\downarrow t$,
the ordering gives $\lim M(s)\geq M(t)$, whereas the maximality
in \eqref{eq:absorption-argmin} gives $\lim M(s)\leq M(t)$.
This proves right continuity, also at $t=0$, where $B$ has the
unique minimizer $m_0$ on $[m_0,m]$. For $s\uparrow t$, the same
ordering bounds $M(s)$ above by every minimizer of $A_t$.
Its limit is itself a minimizer, and hence equals the smallest one.
This proves \eqref{eq:absorption-left-limit}.

\emph{Step 3. Jumps and the finite-speed bound.}
By \eqref{eq:absorption-argmin} and \eqref{eq:absorption-left-limit},
continuity of absorption is equivalent to coincidence of the extreme
minimizers, and a jump has size $\beta-\alpha$.
At a jump time,
$A_t(\alpha)=A_t(\beta)\leq A_t(a)$ for
$\alpha\leq a\leq\beta$; taking differences gives
\eqref{eq:absorption-jump-balance}. This argument does not require
the minimizer set to be an interval.

For the upper bound in \eqref{eq:absorption-finite-speed}, put
$k=\mu_0([0,Lt])$. If $M(t)>k$, choose $a\in(k,M(t))$.
The quantile satisfies $R_0>Lt$ almost everywhere on $(a,M(t))$,
so
\[
 A_t(M(t))-A_t(a)
 =\int_a^{M(t)}(R_0+tU_0)\,\dd b
 \geq\int_a^{M(t)}(R_0-Lt)\,\dd b>0,
\]
contrary to the minimality of $M(t)$. The lower bound is already
contained in \eqref{eq:absorption-argmin}.
\end{proof}

The largest minimizer includes labels arriving exactly at the
absorption time. The minimizers of $A_t$ need not form an interval: the primitive may be strictly above its minimum between them,
although the intervening labels lie at zero after projection.

\begin{remark}[The initial atom in the left limit]\label{rem:origin-atom-left-trace}
The interval $[m_0,m]$ in \eqref{eq:absorption-left-limit} retains the
mass already at the origin. For $\mu_0=\alpha\delta_0+\beta\delta_1$,
$u_0(0)=0$ and $u_0(1)=-1$, with $\alpha,\beta>0$, one has
$m_0=\alpha$ and
\[
 A_t(a)=0\quad(0\leq a\leq\alpha),\qquad
 A_t(a)=(a-\alpha)(1-t)\quad(\alpha\leq a\leq\alpha+\beta).
\]
At $t=1$ every label minimizes $A_1$, but $M(1-)=\alpha$ and
$M(1)=\alpha+\beta$. Taking the smallest minimizer over $[0,m]$ would
incorrectly include the initial atom in the jump.
\end{remark}

The following corollaries give timing and rate criteria; the asymptotic
formula is also used to compute the total impulse in
Theorem~{\ref{thm:boundary-momentum}}.

The endpoint characterization also determines when absorption
begins and when it is complete. Both questions reduce to comparing
the primitive with its values at the endpoints of the label interval.

\begin{corollary}[Onset and completion times]\label{cor:absorption-times}
Assume $m_0<m$.
\begin{enumerate}
\item[\textup{(i)}] Define the onset time by
\begin{equation}\label{eq:absorption-onset}
 T_{\mathrm{on}}=
 \inf_{\substack{m_0<a\leq m\\ C(a)<0}}
       \frac{B(a)}{-C(a)},
 \qquad \inf\varnothing=+\infty.
\end{equation}
Then $M(t)=m_0$ for $0\leq t<T_{\mathrm{on}}$. If
$T_{\mathrm{on}}<\infty$, then $M(t)>m_0$ for
$t>T_{\mathrm{on}}$. If $0<T_{\mathrm{on}}<\infty$, the mass at
the onset time is
\begin{equation}\label{eq:absorption-onset-equality}
 M(T_{\mathrm{on}})=
 \max\left(\{m_0\}\cup
 \left\{a>m_0:C(a)<0,\ 
          \frac{B(a)}{-C(a)}=T_{\mathrm{on}}\right\}\right).
\end{equation}
In this case onset is a jump if and only if the infimum is attained
at a label greater than $m_0$.

\item[\textup{(ii)}] For $m_0\leq a<m$, set
\[
 Q(a)=
 \begin{cases}
 \displaystyle\frac{\int_a^mR_0(b)\,\dd b}
                    {-\int_a^mU_0(b)\,\dd b},
          &\displaystyle\int_a^mU_0(b)\,\dd b<0,\\[6pt]
 +\infty,&\displaystyle\int_a^mU_0(b)\,\dd b\geq0.
 \end{cases}
\]
The completion time is
\begin{equation}\label{eq:absorption-completion}
 T_{\mathrm{all}}=\sup_{m_0\leq a<m}Q(a).
\end{equation}
If $T_{\mathrm{all}}<\infty$, then $M(t)=m$ if and only if
$t\geq T_{\mathrm{all}}$. If $T_{\mathrm{all}}=\infty$,
complete absorption does not occur at a finite time.
\end{enumerate}
\end{corollary}

\begin{proof}
For $a>m_0$, strict increase of $B$ gives $B(a)>0$, and hence,
for $t\geq0$,
\[
 A_t(a)\leq0
 \quad\Longleftrightarrow\quad
 C(a)<0\ \text{ and }\ t\geq\frac{B(a)}{-C(a)}.
\]
Thus $A_t>0$ on $(m_0,m]$ when $t<T_{\mathrm{on}}$, whereas
$A_t(a)<0$ for some $a>m_0$ when
$t>T_{\mathrm{on}}$ and $T_{\mathrm{on}}<\infty$.
Formula~\eqref{eq:absorption-argmin} gives the corresponding values
$M(t)=m_0$ and $M(t)>m_0$.
At a finite positive onset time, $A_{T_{\mathrm{on}}}\geq0$
and its zeros outside $[0,m_0]$ are exactly the labels attaining
the infimum. This gives \eqref{eq:absorption-onset-equality}.
Since $M(T_{\mathrm{on}}-)=m_0$, it also proves the jump criterion.

At the other endpoint, \eqref{eq:absorption-argmin} gives
\[
 M(t)=m\quad\Longleftrightarrow\quad
 \int_a^m(R_0+tU_0)\,\dd b\leq0
                  \quad\text{for all }a\in[m_0,m).
\]
For each such $a$, the position integral $\int_a^mR_0\,\dd b$
is strictly positive. The condition $\int_a^m(R_0+tU_0)\,\dd b\leq0$
therefore requires
$\int_a^mU_0\,\dd b<0$ and is then equivalent to $t\geq Q(a)$.
If the velocity integral is nonnegative, no finite $t$ satisfies
it, as encoded by $Q(a)=+\infty$. Taking the supremum over $a$
proves \eqref{eq:absorption-completion}, including equality at the
completion time.
\end{proof}

The large-time limit is governed by the initial-velocity primitive, with
the initial-position primitive selecting the smallest minimizer.

\begin{corollary}[Asymptotic absorbed mass]\label{cor:asymptotic-absorption}
Let $C$ be the initial-velocity primitive in
\eqref{eq:absorption-primitives}. The asymptotic absorbed mass is its
smallest minimizer on $[m_0,m]$:
\begin{equation}\label{eq:absorption-asymptotic}
 \lim_{t\to\infty}M(t)
       =\min\Argmin_{a\in[m_0,m]}C(a).
\end{equation}
\end{corollary}

\begin{proof}
When $m_0<m$, let $c_*=\min\Argmin_{[m_0,m]}C$.
The minimizing properties of $M(t)$ and $c_*$ give
\[
 B(M(t))+tC(M(t))\leq B(c_*)+tC(c_*),\qquad
 C(M(t))\geq C(c_*).
\]
Consequently $B(M(t))\leq B(c_*)$, and strict increase of $B$
yields $M(t)\leq c_*$. The monotone limit
$\ell=\lim_{t\to\infty}M(t)$ therefore satisfies $\ell\leq c_*$.
On the other hand, $t^{-1}A_t\to C$ uniformly on $[m_0,m]$,
so the minimizing inequality for $M(t)$ passes to
$C(\ell)\leq C(a)$ for every $a\in[m_0,m]$.
The definition of $c_*$ then gives $\ell\geq c_*$.
If $m_0=m$, both sides of \eqref{eq:absorption-asymptotic} equal $m$.
\end{proof}

By Theorem~\ref{thm:absorption}, a unique minimizer yields continuity
of absorption, but a rate
requires local regularity and nondegeneracy. Under these additional
assumptions, the minimizing label can be followed by the implicit
function theorem.

\begin{corollary}[A regular absorption rate]\label{cor:regular-absorption}
Suppose that, for some $t_0>0$, the primitive $A_{t_0}$ has a
unique minimizer $a_0\in(m_0,m)$. Assume that $R_0$ and $U_0$
admit $C^1$ representatives near $a_0$ and satisfy
\[
 R_0'(a_0)+t_0U_0'(a_0)>0.
\]
Then $M$ is $C^1$ in a neighborhood of $t_0$, where
\begin{equation}\label{eq:regular-absorption-rate}
 M'(t)=\frac{-U_0(M(t))}
                  {R_0'(M(t))+tU_0'(M(t))}>0.
\end{equation}
\end{corollary}

\begin{proof}
The interior minimum satisfies $R_0(a_0)+t_0U_0(a_0)=0$.
The implicit function theorem provides a $C^1$ curve $a(t)$
through $a_0$ such that $R_0(a(t))+tU_0(a(t))=0$.
Choose a fixed interval around $a_0$, contained in $(m_0,m)$,
on which $R_0'+tU_0'>0$ for all nearby $t$.
On this interval $A_t$ is strictly convex and $a(t)$ is its
unique minimizer. On the compact complement, the strict gap
above $A_{t_0}(a_0)$ persists for nearby $t$, by uniform
convergence of $A_t$ and continuity of $A_t(a(t))$.
Thus $a(t)$ remains the unique global minimizer, and
\eqref{eq:absorption-argmin} identifies it with $M(t)$.

Differentiation along this curve gives
\[
 \bigl(R_0'(M(t))+tU_0'(M(t))\bigr)M'(t)+U_0(M(t))=0.
\]
The coefficient of $M'(t)$ is positive. Since $M(t)>m_0$,
monotonicity and nonnegativity of $R_0$ give $R_0(M(t))>0$;
the stationary equation then yields
$U_0(M(t))=-R_0(M(t))/t<0$. This proves
\eqref{eq:regular-absorption-rate}.
\end{proof}

\subsection[Radial scalar momentum balance and boundary impulse]{Radial scalar momentum balance and boundary impulse}\label{subsec:boundary-momentum}

The temporal description of absorption identifies the labels whose
radial momentum is removed at the origin. A pointwise boundary flux
need not exist, so we describe this contribution by a Stieltjes measure.
The balance is scalar: the full-space vector momentum still satisfies
\eqref{eq:full-momentum-weak}.

Neumann et al.~\cite{NOSS2022} obtain a counterpart of the
total momentum formula in the absorbing branch of a half-line
initial-boundary problem. Nedeljkov et al.~\cite{NNOS2017} describe the
radial momentum balance for shadow waves with boundary traces.

Under the finite-mass and second-moment assumptions of
Section~\ref{subsec:radial-projection}, we derive the radial scalar momentum
balance directly from the projection and minimizer structure. The
velocity projection \eqref{eq:radial-actual-velocity} identifies total
radial momentum with the initial momentum of the unabsorbed labels.
The minimum-value function of $A_t$ is
\begingroup
\begin{equation}\label{eq:minimum-value-function}
 h(t):=\min_{a\in[m_0,m]}A_t(a),\qquad t\geq0.
\end{equation}
\endgroup
 It then yields the weighted Stieltjes
identity, valid for jumps and for both absolutely continuous and
singular continuous absorption.

\begin{theorem}[Radial scalar momentum balance and boundary impulse]\label{thm:boundary-momentum}
Let $(\mu,u)$ be the radial representation of the solution in
Theorem~\ref{thm:main}, under the finite-mass and second-moment
assumptions of Section~\ref{subsec:radial-projection}, with the scalar radial velocity
identified $\mu_t$-almost everywhere by \eqref{eq:radial-actual-velocity}.
Let $B,C$ be the primitives in \eqref{eq:absorption-primitives}, and set
\[
 M(t)=\mu_t(\{0\}),\qquad
 P(t)=\int_{[0,\infty)}u(t,r)\,\mu_t(\dd r).
\]
Then the following properties hold.
\begin{enumerate}
\item[(i)] The total radial momentum $P$ is nondecreasing and right
continuous on $[0,\infty)$ and satisfies, for every $t\geq0$,
\begin{equation}\label{eq:boundary-total-momentum}
 P(t)=C(m)-C(M(t)).
\end{equation}
\item[(ii)] Its Stieltjes measure $\mathcal I=\dd P$ on $(0,\infty)$
satisfies the measure identities and, for every $t\geq0$, the
cumulative identity
\begin{equation}\label{eq:boundary-impulse-measure}
 \begin{gathered}
 0\leq\mathcal I\leq L\,\dd M,
 \qquad t\,\mathcal I=\dd(B\circ M),\qquad L=\|U_0\|_\infty,\\
 B(M(t))=\int_{(0,t]}s\,\mathcal I(\dd s).
 \end{gathered}
\end{equation}
For every $\psi\in C_c^1([0,\infty)\times[0,\infty))$,
 the same radial entropy solution satisfies
the boundary balance of radial scalar momentum
\begin{equation}\label{eq:boundary-momentum-weak}
 \begin{split}
 &\int_0^\infty\!\int_{[0,\infty)}
       (u\psi_t+u^2\psi_r)\,\dd\mu_t\,\dd t
  +\int_{[0,\infty)}u_0(r)\psi(0,r)\,\dd\mu_0\\
 &\hspace{38mm}+\int_{(0,\infty)}\psi(t,0)\,\mathcal I(\dd t)=0.
 \end{split}
\end{equation}
No vanishing condition is imposed on $\psi(t,0)$.
For tests satisfying $\psi(t,0)=0$, this reduces to
{\eqref{eq:radial-momentum-weak}}.
\item[(iii)] For every $t>0$, writing $\alpha=M(t-)$ and $\beta=M(t)$,
\begin{equation}\label{eq:boundary-impulse-jump}
 \mathcal I(\{t\})=-\int_\alpha^\beta U_0(a)\,\dd a
                  =\frac1t\int_\alpha^\beta R_0(a)\,\dd a.
\end{equation}
The total impulse is
\begin{equation}\label{eq:boundary-total-impulse}
 \mathcal I((0,\infty))=-\min_{a\in[m_0,m]}C(a).
\end{equation}
\end{enumerate}
\end{theorem}

\begin{proof}
\emph{Step 1. Total radial momentum and its monotone variation.}
Since $V_t\in H_{R_t}^0$ and $\mathbf1_{\{R_t>0\}}\in H_{R_t}^0$,
the push-forward relation in \eqref{eq:radial-projection} and the
orthogonal velocity projection in \eqref{eq:radial-actual-velocity} give,
for every $t\geq0$,
\[
 \begin{aligned}
 P(t)&=\int_0^m V_t\,\dd a
      =\int_0^m V_t\mathbf1_{\{R_t>0\}}\,\dd a\\
     &=\int_0^m U_0\mathbf1_{\{R_t>0\}}\,\dd a
      =\int_{M(t)}^m U_0\,\dd a.
 \end{aligned}
\]
At $t=0$ the projection leaves $U_0\in H_{R_0}^0$ unchanged; hence
\eqref{eq:boundary-total-momentum} also holds at the initial time.

The minimizing property of $M(s)$ and $M(t)$ yields
\begin{equation}\label{eq:boundary-minimizer-comparison}
 \begin{aligned}
 B(M(s))+sC(M(s))&\leq B(M(t))+sC(M(t)),\\
 B(M(t))+tC(M(t))&\leq B(M(s))+tC(M(s))
 \end{aligned}
 \qquad(0\leq s<t).
\end{equation}
Adding gives $(t-s)(C(M(t))-C(M(s)))\leq0$. Since $M$ is
nondecreasing and $C$ is $L$-Lipschitz,
\[
 0\leq P(t)-P(s)
 =C(M(s))-C(M(t))\leq L(M(t)-M(s)).
\]
Together with the continuity of $C$ and right continuity of $M$,
this proves the monotonicity and right continuity asserted in (i) and
$0\leq\mathcal I\leq L\,\dd M$. Moreover,
$C(M(0))=C(m_0)=0$ and $M(s)\to m_0$ as $s\downarrow0$;
thus $P(0)=C(m)$ and the momentum has no initial impulse jump.

\emph{Step 2. The minimum value and the weighted Stieltjes identity.}
To compute the weighted impulse without differentiating a possibly
jumping minimizing label,
use $h$ defined in
{\eqref{eq:minimum-value-function}}.
Since $M(t)$ minimizes $A_t$,
\[
 h(t)=B(M(t))+t g(t),\qquad g(t)=C(M(t)).
\]
The function $h$ is concave and Lipschitz, being a minimum of affine
functions with slopes in the bounded range of $C$.
The comparison in \eqref{eq:boundary-minimizer-comparison} implies
\[
 g(t)\leq\frac{h(t)-h(s)}{t-s}\leq g(s)
 \qquad(0\leq s<t).
\]
At continuity points of the monotone function $g$, the difference
quotients converge to $g$; consequently $\dd h=g\,\dd t$.
Both $g$ and $B\circ M$ are of bounded variation, so the product
rule on $(0,\infty)$ gives
\[
 \dd(B\circ M)
 =\dd h-\dd(tg)
 =g\,\dd t-(t\,\dd g+g\,\dd t)
 =t\,\dd P.
\]
Integration over $(s,t]$, followed by $s\downarrow0$, proves the
cumulative formula in \eqref{eq:boundary-impulse-measure}, since
$B(M(s))\to B(m_0)=0$.

\emph{Step 3. Atomic and total impulse.}
At a positive time, the two Stieltjes increments are
\[
 \mathcal I(\{t\})=C(\alpha)-C(\beta),\qquad
 t\mathcal I(\{t\})=B(\beta)-B(\alpha),
\]
which gives \eqref{eq:boundary-impulse-jump}.
Writing $M_\infty=\lim_{t\to\infty}M(t)$, we also obtain
\[
 \mathcal I((0,\infty))
 =\lim_{t\to\infty}(P(t)-P(0))
 =-C(M_\infty)
 =-\min_{a\in[m_0,m]}C(a),
\]
where the last equality uses Corollary~\ref{cor:asymptotic-absorption}.

\emph{Step 4. The radial scalar momentum balance.}
For the boundary identity, let $\mathcal L(\psi)$ denote the spacetime
integral and initial term in \eqref{eq:boundary-momentum-weak}. Set
$\vartheta(t)=\psi(t,0)$ and choose
$\chi_R(r)=\chi(r/R)$, where $\chi\in C_c^1([0,\infty))$ equals
one near zero. The test $\psi-\vartheta\chi_R$ vanishes at $r=0$.
Applying \eqref{eq:radial-momentum-weak} to this test gives
\[
 \begin{aligned}
 \mathcal L(\psi)=\mathcal L(\vartheta\chi_R)
 &=\int_0^\infty \vartheta'(t)\int u\chi_R\,\dd\mu_t\,\dd t
   +\int_0^\infty \vartheta(t)\int u^2\chi_R'\,\dd\mu_t\,\dd t\\
 &\quad+\vartheta(0)\int u_0\chi_R\,\dd\mu_0.
 \end{aligned}
\]
If $\vartheta$ is supported in $[0,T]$, the absolute value of the flux
term is at most
\[
 R^{-1}\|\chi'\|_\infty L^2mT\|\psi(\cdot,0)\|_\infty,
\]
which tends to zero. The other terms converge by dominated
convergence, since the mass is finite and $|u|\leq L$.
Passing to the limit and integrating the right-continuous BV
function $P$ by parts gives
\[
 \mathcal L(\psi)
 =\int_0^\infty P(t)\vartheta'(t)\,\dd t+P(0)\vartheta(0)
 =-\int_{(0,\infty)}\vartheta(t)\,\mathcal I(\dd t).
\]
This is \eqref{eq:boundary-momentum-weak}.
\end{proof}

This completes the proof of
Theorem~\ref{thm:main-absorption};
Corollary~\ref{cor:impulse-types} below further describes the atomless and
atomic parts of the impulse.

Thus the scalar equation has boundary source $\mathcal I\otimes\delta_0$:
removing the incoming net negative radial momentum increases $P$.
This does not add a source to the full-space vector momentum equation,
since the scalar radial test obtained by angular averaging
of a full-space vector test satisfies $\psi(t,0)=0$.

Arrival times give a common representation of the absorption and
impulse measures. Pushing forward the mass of labels arriving at positive
finite times gives $\dd M$; weighting these labels by $R_0$ gives
$\dd(B\circ M)=t\mathcal I$.
This separates their atomless parts from the contribution of an incoming
label interval at a jump.

\begin{corollary}[Atomless and atomic impulse]\label{cor:impulse-types}
Under the assumptions of Theorem~\ref{thm:boundary-momentum}, let
\[
 \nu=\dd M=\nu_{\mathrm{ac}}+\nu_{\mathrm{sc}}
                 +\sum_{t>0}\Delta M(t)\delta_t
\]
be the Lebesgue decomposition of the absorption measure on
$(0,\infty)$. The atomless parts of the impulse satisfy
\begin{equation}\label{eq:boundary-impulse-types}
 \mathcal I_{\mathrm{ac}}
       =\frac{R_0(M(t))}{t}\nu_{\mathrm{ac}},\qquad
 \mathcal I_{\mathrm{sc}}
       =\frac{R_0(M(t))}{t}\nu_{\mathrm{sc}},
\end{equation}
and its atoms are given by \eqref{eq:boundary-impulse-jump}.
The measure identities in \eqref{eq:boundary-impulse-types} are
unaffected by changes to $R_0$ on label-null sets and by arbitrary
finite assignments at $0$ and $m$. Moreover, $\mathcal I$ and $\dd M$ are
mutually absolutely continuous on $(0,\infty)$.
\end{corollary}

\begin{proof}
Set $M_\infty=\lim_{t\to\infty}M(t)$ and define the arrival map
\[
 \tau(a)=\inf\{t\geq0:M(t)\geq a\},
 \qquad m_0<a<M_\infty.
\]
This map is nondecreasing and takes values in $(0,\infty)$:
right continuity at zero excludes $\tau(a)=0$, and
$a<M_\infty$ ensures finite arrival. The inverse relation gives,
up to label endpoints,
\[
 \{a:s<\tau(a)\leq t\}=(M(s),M(t)]
 \qquad(0\leq s<t).
\]
Consequently
\[
 \begin{aligned}
 \int_{\{s<\tau\leq t\}}\dd a&=M(t)-M(s),\\
 \int_{\{s<\tau\leq t\}}R_0(a)\,\dd a
 &=B(M(t))-B(M(s)).
 \end{aligned}
\]
Equality on these intervals identifies the two Stieltjes measures:
\begin{equation}\label{eq:arrival-pushforwards}
 \dd M=\tau_\#\bigl(\dd a|_{(m_0,M_\infty)}\bigr),\qquad
 \dd(B\circ M)
 =\tau_\#\bigl(R_0(a)\,\dd a|_{(m_0,M_\infty)}\bigr).
\end{equation}

For a jump time $t$, the arrival-time fibre of $\tau$ is
$(M(t-),M(t))$ up to endpoints. Removing all such intervals and
their endpoints leaves the labels for the atomless parts; on those
labels $M(\tau(a))=a$. In the weighted identity in
\eqref{eq:arrival-pushforwards}, one may
therefore replace $R_0(a)$ by $R_0(M(\tau(a)))$, obtaining
\[
 \bigl[\dd(B\circ M)\bigr]_{\mathrm{atomless}}
 =R_0(M(t))\bigl[\dd M\bigr]_{\mathrm{atomless}}.
\]
The weighted push-forward in \eqref{eq:arrival-pushforwards} also shows
that changing $R_0$ on a label-null set leaves the atomless measure
identity unchanged. Combining the atomless identity with
$t\mathcal I=\dd(B\circ M)$ from \eqref{eq:boundary-impulse-measure} gives
\eqref{eq:boundary-impulse-types}; the arrival-time fibres at jumps give
\eqref{eq:boundary-impulse-jump}.
Since $R_0>0$ almost everywhere on $(m_0,m)$, the two label measures
in \eqref{eq:arrival-pushforwards} have the same null sets. Their
push-forwards are therefore mutually absolutely continuous, as are
$\mathcal I$ and $\dd M$ after multiplication by $1/t$ on
$(0,\infty)$.
\end{proof}

At an absorption jump, the impulse involves an average over the
whole incoming label interval. It cannot in general be replaced by
$-U_0(M(t))\Delta M(t)$. The atomless formulas likewise require no
pointwise representative of $U_0$ at the label $M(t)$.

\subsection{Stability of absorbed mass}\label{subsec:absorption-stability}

Convergence of projected positions alone does not control the mass at
the origin. For this observable we use the temporal minimizer formula.
Uniform convergence of the primitives and convergence of the initial
origin masses control the moving minimization intervals.
A unique limiting minimizer gives convergence at positive
continuity times; the extreme minimizers give the two-sided bound at jump
times.

\begin{proof}[Proof of Proposition~\ref{thm:main-stability}]
\emph{Step 1. Uniform convergence of the primitives.}
Write $M_k(t)=\mu_t^k(\{0\})$ and $M(t)=\mu_t(\{0\})$.
For $A_t^k(a)=\int_0^a(R_0^k+tU_0^k)\,\dd b$,
Cauchy--Schwarz gives uniform convergence of the primitives on every
bounded time interval:
\begin{equation}\label{eq:primitive-uniform-convergence}
 \sup_{0\leq t\leq T}\|A_t^k-A_t\|_{C([0,m])}
 \leq\sqrt m\bigl(\|R_0^k-R_0\|_{L^2(0,m)}
                 +T\|U_0^k-U_0\|_{L^2(0,m)}\bigr)
 \longrightarrow0.
\end{equation}

\emph{Step 2. Moving minimization intervals and the two-sided bound.}
Fix $t>0$ and a subsequence for which $M_k(t)\to b$.
The inequalities $M_k(t)\geq m_0^k$ imply $b\in[m_0,m]$.
Given $a\in[m_0,m]$, the labels
$a_k=\max\{a,m_0^k\}$ belong to the corresponding minimization
intervals and converge to $a$. Thus
\[
 A_t(b)
 =\lim_{k\to\infty}A_t^k(M_k(t))
 \leq\lim_{k\to\infty}A_t^k(a_k)
 =A_t(a).
\]
The equalities use \eqref{eq:primitive-uniform-convergence} and
continuity of $A_t$; the inequality is the minimizing property of
$M_k(t)$. Every cluster point consequently lies in
$\Argmin_{[m_0,m]}A_t$. By \eqref{eq:absorption-left-limit} and
\eqref{eq:absorption-argmin}, the extreme
elements of this set are $M(t-)$ and $M(t)$, which proves
the two-sided bound in Proposition~\ref{thm:main-stability}. At positive
continuity times $M(t-)=M(t)$, so $M_k(t)\to M(t)$; at zero the
required convergence is
$m_0^k\to m_0$.

\emph{Step 3. Convergence of the Stieltjes absorption measures.}
Since $M$ is monotone, its discontinuity set is at most countable.
The pointwise convergence just proved and $0\leq M_k\leq m$
therefore give, for every $\vartheta\in C_c^1((0,\infty))$,
\[
 \begin{aligned}
 \int\vartheta\,\dd M_k
 &=-\int_0^\infty\vartheta'(t)M_k(t)\,\dd t\\
 &\longrightarrow-\int_0^\infty\vartheta'(t)M(t)\,\dd t
 =\int\vartheta\,\dd M.
 \end{aligned}
\]
Here the identities are Stieltjes integration by parts and the
limit is dominated convergence. The common bound
$\dd M_k((0,\infty))\leq m$ extends this convergence, by uniform
approximation, to every $\vartheta\in C_c((0,\infty))$.

\end{proof}

When $m_0=0$, the initial-origin-mass condition follows already
from the convergence of the quantiles: $L^2$ convergence of $R_0^k$
alone implies
$\limsup_k m_0^k\leq m_0$. Indeed, if $m_0^k\geq c>m_0$ along a
subsequence, then $R_0^k=0$ almost everywhere on $(0,c)$ and
\[
 0<\int_0^c R_0(a)^2\,\dd a
 =\int_0^c|R_0^k(a)-R_0(a)|^2\,\dd a
 \leq\|R_0^k-R_0\|_{L^2(0,m)}^2\longrightarrow0,
\]
where strict positivity follows from $|\{R_0=0\}|=m_0<c$.
Thus $m_0=0$ forces $m_0^k\to0$.

The representation \eqref{eq:radial-projection} also gives, for every $T<\infty$,
\begin{equation}\label{eq:radial-position-stability}
 \sup_{0\leq t\leq T}\|R_t^k-R_t\|_{L^2(0,m)}
 \leq\|R_0^k-R_0\|_{L^2(0,m)}
       +T\|U_0^k-U_0\|_{L^2(0,m)},
\end{equation}
by nonexpansivity of $P_{\Kplus}$.
The stationary-shell example in the introduction shows
that this estimate alone does not control the mass of the zero level set.
The initial-origin-mass hypothesis
in Proposition~\ref{thm:main-stability} cannot in general be omitted.

\subsection[Examples and prescribed absorption measures]{Examples and prescribed absorption measures}\label{subsec:absorption-examples}

The following examples distinguish absolutely continuous absorption from a shell
collision followed by a jump at the origin. The final construction
prescribes the absorption measure, including singular continuous absorption.
Spatial mass is recovered
from the physical radial mass by \eqref{eq:physical-radial-mass}.

\begin{example}[Inward transport of a smooth density]\label{ex:inward}
Let $f\in C_c^\infty((1,2))$ be nonnegative and nonzero,
$\mu_0(\dd r)=\omega_n r^{n-1}f(r)\dd r$, where
$\omega_n=|\mathbb S^{n-1}|$, and let $u_0=-1$ on its support.
The radial map from initial radius $s$ is $(s-t)_+$. Hence
\[
 M(t)=\omega_n\int_0^t s^{n-1}f(s)\dd s,\qquad
 \mu_t|_{(0,\infty)}=\omega_n(r+t)^{n-1}f(r+t)\dd r.
\]
The absorbed atom is stationary and the remaining mass has velocity
$-1$. Absorption is absolutely continuous, with
$M'(t)=\omega_n t^{n-1}f(t)$.
If $f(t)>0$, the unweighted spatial radial density is
\[
 f_t(r):=\frac{(r+t)^{n-1}}{r^{n-1}}f(r+t)
 \sim t^{n-1}f(t)r^{1-n}\quad(r\downarrow0).
\]
For $n\ge2$ this unweighted measure $f_t(r)\dd r$ in the
radius variable is not locally finite at zero;
the physical radial mass $\mu_t$ and its spatial reconstruction are.
\end{example}

\begin{example}[A shell collision followed by absorption]\label{ex:collision}
Take two unit masses at radii $2$ and $4$ with velocities $-1$ and $-3$.
They collide at time $1$ and radius $1$. The resulting shell has mass
$2$, velocity $-2$, and reaches the origin at time $3/2$. Thus
\[
 \mu_t=
 \begin{cases}
  \delta_{2-t}+\delta_{4-3t},&0\le t<1,\\
  2\delta_{3-2t},&1\le t<3/2,\\
  2\delta_0,&t\ge3/2.
 \end{cases}
\]
At $t=3/2$, $A_t$ attains its minimum value at $a=0$ and $a=2$ and is
strictly larger between them. Formula~\eqref{eq:absorption-left-limit} gives
$M(3/2-)=0$, whereas \eqref{eq:absorption-argmin} gives $M(3/2)=2$.
Thus the minimizer set need not be an interval. The total radial momentum is $-4$
until time $3/2$ and is zero thereafter; hence
$\mathcal I=4\delta_{3/2}$. The positive-radius collision contributes
no boundary impulse.
\end{example}

As a complement to these forward formulas, one can fix a smooth initial
mass and choose a bounded Borel velocity to prescribe its absorption
measure. The stopped-trajectory form also appears in
Hynd--Tudorascu~\cite[(2.4)--(2.5)]{HT2024}. Here
ordered arrival times make the resulting free-position map nondecreasing
wherever it is positive, so its positive part belongs to
$\Kplus$ and the projection reduces to truncation at zero.

\begin{proposition}[Prescribed absorption measures]\label{prop:absorption-realization}
Let $\mu_0=f(r)\dd r$, where $f\in C_c^\infty((1,2))$ is
nonnegative and $\int_1^2f(r)\,\dd r=1$, and let $R_0$ be its
quantile map. Given a probability measure $\nu$ supported in
$[T_-,T_+]$, with $0<T_-\leq T_+<\infty$, let
$T:(0,1)\to[T_-,T_+]$ denote its nondecreasing quantile and define
\begin{equation}\label{eq:prescribed-arrival-velocity}
 U_0(a)=-\frac{R_0(a)}{T(a)}.
\end{equation}
\begin{samepage}
This velocity has the form $U_0=u_0\circ R_0$ almost everywhere for
a bounded Borel radial velocity $u_0$ satisfying $u_0(0)=0$.
For every $t\geq0$, the radial representation of the corresponding
entropy solution has the following Lagrangian position, right derivative in $L^2(0,1)$,
and absorption law:
\begin{align}
 R_t(a)&=R_0(a)\left(1-\frac{t}{T(a)}\right)_+,
 &V_t(a)&=U_0(a)\mathbf1_{\{T(a)>t\}},
       \label{eq:prescribed-arrival-map}\\
 M(t)&=\nu((0,t]),
 &\dd M&=\nu.
       \label{eq:prescribed-absorption-law}
\end{align}
The position and velocity identities hold in $L^2(0,1)$, and the
Stieltjes identity holds on $(0,\infty)$.
\end{samepage}
\end{proposition}

\begin{proof}
\emph{Step 1. Compatibility with the initial mass.}
The distribution function $F_0(r)=\int_0^r f(s)\,\dd s$ satisfies
$F_0(R_0(a))=a$ for almost every $a\in(0,1)$, since $\mu_0$ is
atomless. Define the radial velocity by
\[
 u_0(r)=
 \begin{cases}
 -r/T(F_0(r)),&1\leq r\leq2\text{ and }0<F_0(r)<1,\\
 0,&\text{otherwise}.
 \end{cases}
\]
The functions $F_0$ and $T$ are Borel and $T\geq T_->0$; hence
$u_0$ is Borel, $|u_0|\leq2/T_-$, and $u_0(0)=0$.
Pulling back by $R_0$ gives
\[
 u_0(R_0(a))
 =-\frac{R_0(a)}{T(F_0(R_0(a)))}
 =-\frac{R_0(a)}{T(a)}
 =U_0(a)
 \quad\text{for almost every }a\in(0,1).
\]
Thus the prescribed Lagrangian velocity is compatible with the
initial mass, which is atomless.

\emph{Step 2. Identification of the position projection.}
For a fixed $t\geq0$, write $F_t=R_0+tU_0$. Its positive part is
\[
 (F_t)_+=R_0\left(1-\frac tT\right)_+\in\Kplus,
\]
because both factors are nonnegative and nondecreasing. Pointwise
projection onto $[0,\infty)$ gives, for every $R\in\Kplus$,
\[
 \|F_t-(F_t)_+\|_{L^2(0,1)}^2
 \leq\|F_t-R\|_{L^2(0,1)}^2.
\]
Since $(F_t)_+$ itself belongs to $\Kplus$, it equals
$P_{\Kplus}F_t$. Theorem~\ref{thm:projection} therefore identifies
it with $R_t$.

\emph{Step 3. The right derivative and absorption law.}
For $h>0$, the stopped linear trajectories satisfy
\[
 \left|\frac{R_{t+h}(a)-R_t(a)}{h}\right|\leq|U_0(a)|,
 \qquad
 \frac{R_{t+h}(a)-R_t(a)}{h}
 \longrightarrow U_0(a)\mathbf1_{\{T(a)>t\}}.
\]
The pointwise limit is zero also when $T(a)=t$. Since
$U_0\in L^2(0,1)$, dominated convergence proves the right-derivative
identity in \eqref{eq:prescribed-arrival-map}.
Finally, $R_0>0$ almost everywhere and $T_\#\dd a=\nu$, so
\[
 M(t)=|\{a:R_t(a)=0\}|
     =|\{a:T(a)\leq t\}|
     =\nu((0,t]).
\]
The right continuity of $M$ and $M(0)=0$ then identify its
Stieltjes measure on $(0,\infty)$ with $\nu$.
\end{proof}

For example, if $\nu_C$ is the Cantor probability measure on $[0,1]$
and $S(s)=2+s$, choosing $\nu=S_\#\nu_C$ gives purely singular
continuous absorption with the same smooth initial mass. Choosing a
mixture instead gives
\begin{equation}\label{eq:three-absorption-components}
 \dd M=\nu=\tfrac13\mathbf1_{(2,3)}(t)\,\dd t
                  +\tfrac13 S_\#\nu_C+\tfrac13\delta_{5/2}.
\end{equation}
Thus all three parts of the absorption measure can be nonzero for a
single solution; no smoothness of the initial velocity is asserted.

\begingroup
\renewcommand{\addcontentsline}[3]{}
\section*{Use of generative AI}

During the preparation of this manuscript, the authors used ChatGPT (OpenAI) to assist with language editing, \LaTeX{} formatting, and the organization, review and clarification of the mathematical exposition. The authors independently assessed and verified the suggestions incorporated into the manuscript and take full responsibility for its content, including the mathematical statements, computations, proofs and references.

\par
\endgroup

\end{document}